\documentclass[12pt,reqno]{amsart}
\usepackage[latin9]{inputenc}
\usepackage{geometry}
\usepackage{color}
\usepackage{mathrsfs}
\usepackage{amsbsy}
\usepackage{amsthm}
\usepackage{amssymb}
\usepackage{graphicx}
\usepackage[unicode=true,
 bookmarks=false,
 breaklinks=false,pdfborder={0 0 1},backref=page,colorlinks=true]{hyperref}
\usepackage{amsmath,amsthm,amssymb,amscd,mathrsfs,a4wide,comment,dsfont,mathrsfs,setspace,color,comment,float,bbm}
\usepackage{tikz}
\usepackage{pgfplots}
\definecolor{skyblue}{rgb}{0.85,0.85,1}

\makeatletter
\numberwithin{equation}{section}
\numberwithin{figure}{section}
\theoremstyle{plain}
\newtheorem{thm}{Theorem}[section]
\theoremstyle{definition}
\newtheorem{defn}[thm]{Definition}
\theoremstyle{plain}
\newtheorem{prop}[thm]{Proposition}
\theoremstyle{plain}
\newtheorem{cor}[thm]{Corollary}
\theoremstyle{plain}
\newtheorem{lem}[thm]{Lemma}
\newtheorem{rem}[thm]{Remark}

\@ifundefined{date}{}{\date{}}

\usepackage{amsthm}
\usepackage{amscd}
\usepackage{comment}
\usepackage{mathrsfs}
\usepackage{setspace}
\usepackage{color}

\DeclareMathOperator{\grad}{grad}

\DeclareMathOperator{\Hess}{Hess}

\DeclareMathOperator{\dual}{dual}

\DeclareMathOperator{\dv}{div}
\DeclareMathOperator{\Min}{Min}
\DeclareMathOperator{\Max}{Max}
\DeclareMathOperator{\Sd}{Sad}
\DeclareMathOperator{\supp}{supp}

\newcommand{\rme}{\mathrm e}
\newcommand{\Id}{\mathbb I}
\newcommand{\Nd}{\mathcal N}

\newcommand{\mc}{\mathcal}

\newcommand{\rz}{\mathbb R}
\newcommand{\R}{\mathbb R}

\newcommand{\bs}{\boldsymbol}
\newcommand{\be}{\begin{equation}}
\newcommand{\ee}{\end{equation}}
\newcommand{\ba}{\begin{aligned}}
\newcommand{\ea}{\end{aligned}}

\newcommand{\pO}{\partial\Omega}
\newcommand{\DN}{\Delta^{N}_\Omega}
\newcommand{\cupdot}{\mathbin{\mathaccent\cdot\cup}}
\begin{document}

\title{Defining the spectral position of a Neumann domain}
\author{Ram Band$^{1}$, Graham Cox$^{2}$, Sebastian K. Egger$^{1}$}
\address{$^{1}${\small{}Department of Mathematics, Technion--Israel Institute
of Technology, Haifa 32000, Israel}}
\address{$^{2}${\small{}Department of Mathematics and Statistics, Memorial University, St. John's, NL A1C 5S7, Canada}}
\begin{abstract}
A Laplacian eigenfunction on a two-dimensional Riemannian manifold
provides a natural partition into Neumann domains (a.k.a. a Morse--Smale
complex). This partition is generated by gradient flow lines of
the eigenfunction, which bound the so-called Neumann domains.
We prove that the Neumann Laplacian defined on a Neumann
domain is self-adjoint and has a purely discrete
spectrum. In addition, we prove that the restriction of an eigenfunction
to any one of its Neumann domains is an eigenfunction
of the Neumann Laplacian. 
By comparison, similar statements about the Dirichlet Laplacian on a nodal domain
of an eigenfunction are basic and well-known.
The difficulty here is that the boundary of a Neumann domain may have cusps and cracks,
so standard results about Sobolev spaces 
are not available.
Another very useful common fact is that the restricted eigenfunction
on a nodal domain is the first eigenfunction of the Dirichlet
Laplacian. This is no longer true for a Neumann domain. Our results enable the investigation of the resulting spectral
position problem for Neumann domains, which is much more involved
than its nodal analogue.
\end{abstract}

\maketitle

\tableofcontents

\section{Introduction and statement of results}
Let $M$ be a closed, connected, orientable surface with
a smooth Riemannian metric $g$.
It is well-known that the Laplace--Beltrami operator $\Delta$ is self-adjoint and 
has a purely discrete spectrum. We arrange the eigenvalues in increasing
order
\begin{equation}
0=\lambda_{0}<\lambda_{1}\leq\lambda_{2}\leq\cdots,
\end{equation}
and let $\{f_{n}\}_{n=0}^{\infty}$ denote a corresponding complete
system of orthonormal eigenfunctions, so that
\begin{equation}
\Delta f_{n}=\lambda_{n}f_{n}.
\end{equation}

While we are motivated by the study of eigenfunctions, most of the results and constructions in this paper are valid for arbitrary Morse functions. It is well known that for a generic Riemannian metric all of the Laplace--Beltrami eigenfunctions are Morse \cite{Uhl_ajm76}. 

The main objects of study in this paper are the Neumann domains of
a Morse function, to be defined next. Given a smooth function $f$
on $M$, we let $\varphi:\mathbb{R}\times M\rightarrow M$ denote
the flow along the gradient vector field, i.e. the solution to
\begin{equation}
\partial_{t}\varphi(t, x)=-\grad f\big|_{\varphi(t, x)},\qquad\varphi(0, x) = x.\label{eq:flow-1}
\end{equation}
For a critical point $\bs c$ of $f$, we define its stable and unstable
manifolds by
\begin{equation}
\begin{aligned}W^{s}(\bs c) & :=\Big\{x \in M:\lim_{t\rightarrow\infty}\varphi(t,x) = \bs c \Big\},\\
W^{u}(\bs c) & :=\Big\{x \in M:\lim_{t\rightarrow-\infty}\varphi(t,x)=\bs c \Big\}.
\end{aligned}
\label{eq:Stable-Unstable-Def-1}
\end{equation}
We denote the sets of minima,
maxima and saddles of $f$ by $\Min(f)$, $\Max(f)$ and $\Sd(f)$,
respectively.
\begin{defn}[\cite{band:2016}]
\label{def:Neumann-Domains-and-Lines-1} Let $f$ be a Morse function
on $M$.
\begin{enumerate}
\item Let $\bs p\in\Min(f)$ and $\bs q\in\Max(f)$, such that $W^{s}\left(\bs p\right)\cap W^{u}\left(\bs q\right)\neq\emptyset$.
Each of the connected components of $W^{s}\left(\bs p\right)\cap W^{u}\left(\bs q\right)$
is called a \emph{Neumann domain} of $f$.
\item The \emph{Neumann line set} of $f$ is
\begin{equation}
\Nd:={\overline{\bigcup_{{\bs r\in\Sd(f)}}W^{s}(\bs r)\cup W^{u}(\bs r)}}.\label{eq:Neumann-line-set-1}
\end{equation}
\end{enumerate}
\end{defn}

This defines a partition of the manifold $M$, which we call
the Neumann partition. It is not hard to show that $M$ equals
the disjoint union of all Neumann domains and the Neumann line set, under the assumption that $\Nd\neq\emptyset$, see \cite[Proposition 1.3]{band:2016}. (Note that $\Nd = \emptyset$ means $f$ has no saddle points; this is only 
possible when $M$ is a sphere and $f$ has exactly two critical points.)
Figure \ref{fig:example} depicts the Neumann partition of a particular
eigenfunction on the flat torus.

By construction we have that $\grad f$ is parallel to the boundary of any Neumann domain $\Omega$, as the boundary is made up of gradient flow lines, so we conclude that the normal derivative vanishes, $\partial_\nu f\big|_{\pO} = 0$, assuming $\partial\Omega$ is sufficiently smooth. This formal observation motivates our study of the Neumann Laplacian on $\Omega$, which we precisely define in Definition~\ref{def:NeuLap}.

\begin{figure}[ht]
\centering{}\includegraphics[width=1\textwidth]{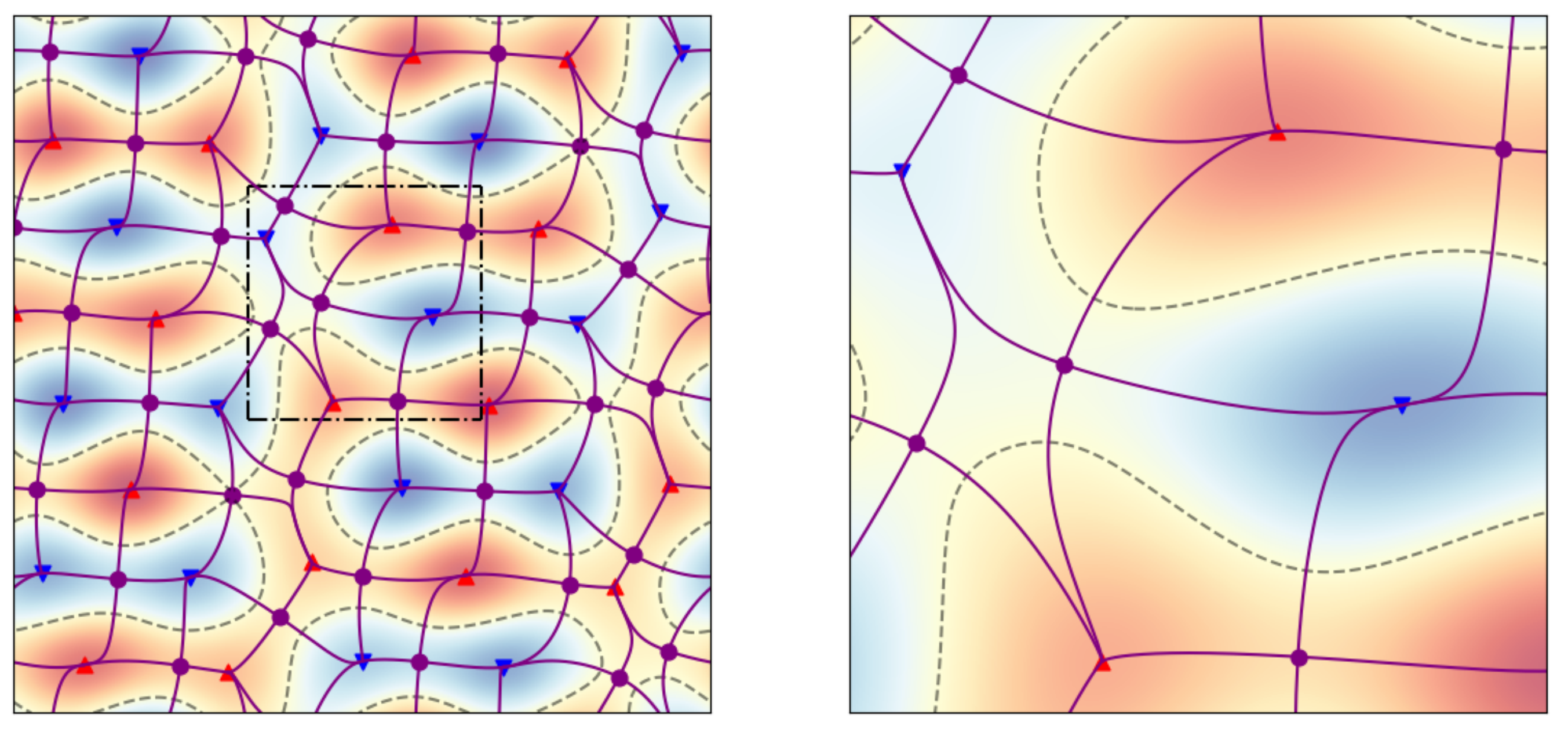} \caption{Left: An eigenfunction corresponding to the eigenvalue $\lambda=17$ on
the flat torus with fundamental domain $[0,2\pi]\times[0,2\pi]$.
Circles mark saddle points and triangles mark extremal points
(maxima by triangles pointing upwards and vice verse for minima).
The nodal set is marked by dashed lines and the Neumann line set 
by solid lines. The Neumann domains are the domains bounded
by the Neumann line set. Right: A magnification of the marked square
from the left figure, showing Neumann domains with and without cusps. 
(This figure was produced using \cite{python_Taylor}.)}
\label{fig:example}
\end{figure}

While the Dirichlet Laplacian on any bounded open set has a purely
discrete spectrum, the same is not necessarily true of the Neumann
Laplacian. Indeed, the essential spectrum may be nonempty, and in
fact can be an arbitrary closed subset of $[0,\infty)$, see \cite{HSS91}.
Nevertheless, the Neumann Laplacian of a Neumann domain is well behaved.
\begin{thm}
\label{thm:discrete} Let $\Omega$ be a Neumann domain of a Morse
function $f$. Then the Neumann Laplacian $\DN$ on
$\Omega$ is a non-negative, self-adjoint operator with purely discrete
spectrum, i.e. consisting only of isolated eigenvalues of finite multiplicity.
\end{thm}

The main difficulty in proving this theorem is due to possible cusps on
the boundary of the Neumann domain; see Proposition
\ref{prop:angles-at-critical-pts} and the preceding discussion.
Such cusps prevent the application of standard results on density
and compact embeddings of Sobolev spaces. We overcome this difficulty
in the proof of Theorem \ref{thm:discrete} by using some particular 
geometric properties that the Neumann domain boundary possesses.

It is well known that the restriction of $f$ to any of its nodal
domains is an eigenfunction of the Dirichlet Laplacian. Similarly,
we have

\begin{thm}
\label{thm:restriction} If $\Omega$ is a Neumann domain of a Morse
function $f$, then $f\big|_{\Omega}\in\mc D(\DN)$. In particular, if $f$ is an eigenfunction of $\Delta$, then $f\big|_{\Omega}$ is an eigenfunction of $\DN$ with the same eigenvalue.
\end{thm}

In fact, we prove much more: in Proposition \ref{prop:domain_of_Laplacian} we completely characterize the domain of the Neumann Laplacian, and in Proposition \ref{prop:intlimit} and Corollary \ref{cor:intlimit} we give some easily verified sufficient conditions for a function to be in the domain of $\DN$. 

The importance of Theorem \ref{thm:discrete} is that it allows us to define the spectral position of a Morse eigenfunction, by which we mean an eigenfunction of the Laplace--Beltrami operator that is also a Morse function.

\begin{defn}
\label{def:Spectral-Position} Let $f$ be a Morse eigenfunction for
an eigenvalue $\lambda$, and let $\Omega$ be a Neumann domain of
$f$. We define the \emph{spectral position} of $\Omega$ as the position
of $\lambda$ in the Neumann spectrum of $\Omega$, i.e.
\begin{equation}
N_{\Omega}(\lambda):=\left|\left\{ \mu_{n}\in\sigma(\DN)~:~\mu_{n}<\lambda\right\} \right|,\label{eq:Spectral-Position-1}
\end{equation}
where $\sigma(\Omega):=\{\mu_{n}\}_{n=0}^{\infty}$ is the Neumann
spectrum of $\Omega$ (which is discrete by Theorem \ref{thm:discrete}),
containing multiple appearances of degenerate eigenvalues and including
$\mu_{0}=0$.
\end{defn}

From Theorem \ref{thm:restriction} we in fact have $\lambda = \mu_n$ for some $n$, and so we can equivalently write
\[
	N_\Omega(\lambda) = \min \{n : \mu_n = \lambda\}.
\]
In particular, if $\lambda \in \sigma(\DN)$ is simple, then $\lambda = \mu_n$ for a unique $n$, and hence $N_\Omega(\lambda) = n$. This equality explains the terminology ``spectral position" for  $N_\Omega(\lambda)$.

The spectral position is a key notion for Neumann domains. Finding
its value is a great challenge and is of major importance in studying
Neumann domains and their properties \cite{band:2016,band2017ground,abbe:2018}.
The corresponding notion for a \emph{nodal} domain is trivial:
if $D$ is a nodal domain of $f$, then $f|_D$ is always the first
eigenfunction of the Dirichlet Laplacian on $D$. This is a basic
observation which serves as an essential ingredient in many nodal
domain proofs. No such result holds for Neumann domains, and in fact the spectral position of an eigenfunction restricted to a Neumann domain can be arbitrarily high, by \cite[Thm.~1.4]{band2017ground}.

\subsection*{Structure of the paper}

In Section \ref{sec: Geometric-properties-of-Neumann-Domains} we
describe some essential geometric properties of Neumann domains, emphasizing
the potentially singular nature of their boundary. In Section \ref{sec:Sobolev}
we use this geometric structure to establish fundamental properties
of Sobolev spaces on Neumann domains, including non-standard density
and compactness results. Finally, in Section \ref{sec:NeumannLaplacian}
we use these properties to study the Neumann Laplacian, in particular 
proving Theorems \ref{thm:discrete} and \ref{thm:restriction}.

\section{Geometric properties of Neumann domains}
\label{sec: Geometric-properties-of-Neumann-Domains}

As above, we take $M$ to be a closed, connected,
orientable Riemannian surface with a smooth Riemannian metric $g$.
Note that all of the statements in this section hold for arbitrary
Morse functions, and not only for eigenfunctions.  For convenience we recall the following definitions.
\begin{defn}
Let $f:M\rightarrow\rz$ be a smooth function.
\begin{enumerate}
\item $f$ is said to be a \emph{Morse function} if the Hessian, $\Hess f(\boldsymbol{p})$, is non-degenerate at every critical point $\bs{p}$ of $f$.
\item A Morse function $f$ is said to be \emph{Morse--Smale} if for all critical points $\bs p$ and $\bs q$,
the stable and unstable manifolds $W^s(\bs p)$ and $W^u(\bs q)$ intersect transversely (cf. Lemma~\ref{prop:Morse-Smale-on-2d} for an equivalent definition in two dimensions).
\end{enumerate}
\end{defn}

We now recall some basic topological properties of Neumann domains.

\begin{thm}
\label{thm:topological-properties-manifolds}\cite[Thm.~1.4]{band:2016}
Let $f$ be a Morse function with a non-empty set of saddle points. Let $\bs p\in\Min(f),\,\bs q\in\Max(f)$ with
$W^{s}\left(\bs p\right)\cap W^{u}\left(\bs q\right)\neq\emptyset$,
 and let $\Omega$ be a connected component of $W^{s}\left(\bs p\right)\cap W^{u}\left(\bs q\right)$,
i.e. a Neumann domain. The following properties hold.
\begin{enumerate}
\item The Neumann domain $\Omega$ is a simply connected open set.\label{enu:thm-topological-properties-manifolds-1}
\item All critical points of $f$ belong to the Neumann line set.\label{enu:thm-topological-properties-manifolds-2}
\item The extremal points which belong to $\overline{\Omega}$ are exactly
$\bs p$ and $\bs q$.\label{enu:thm-topological-properties-manifolds-3}
\item If $f$ is a Morse--Smale function, then $\partial\Omega$ consists
of Neumann lines connecting saddle points with $\bs p$ or $\bs q$.
In particular, $\partial\Omega$ contains either one or two saddle
points. \label{enu:thm-topological-properties-manifolds-4}
\item If $c\in\rz$ is such that $f(\bs p)<c<f(\bs q)$, then $\overline{\Omega}\cap f^{-1}\left(c\right)$
is a smooth, non-self intersecting one-dimensional curve in $\overline{\Omega}$,
with its two boundary points lying on $\partial\Omega$.\label{enu:thm-topological-properties-manifolds-5}
\end{enumerate}
\end{thm}

Parts \eqref{enu:thm-topological-properties-manifolds-2} and \eqref{enu:thm-topological-properties-manifolds-4} of this theorem motivate us to examine individual Neumann lines and their
connectivity to the critical points of $f$.
\begin{defn}
\label{def: Neumann_line_and_degree_of_critical_pt}~
\begin{enumerate}
\item A \emph{Neumann line} is the closure of a connected component
of $W^{s}(\bs r) \setminus\{\bs r\}$ or $W^{u}(\bs r)\setminus\{\bs r\}$ for some $\bs r\in\Sd(f)$.
\item For a critical point $\bs c$ of $f$, we define its degree, $\deg(\bs c)$,
to be the number of Neumann lines connected to $\bs c$.
\end{enumerate}
\end{defn}

Each Neumann line is thus the closure of a gradient flow line, connecting a saddle point to another critical point. In fact, the connectivity of Neumann lines is directly related to the Morse--Smale property of $f$.

\begin{lem}[\cite{abbe:2018}]
\label{prop:Morse-Smale-on-2d}
On a two-dimensional manifold, a Morse
function is Morse--Smale if and only if there is no Neumann line connecting
two saddle points.
\end{lem}

The following properties of Neumann lines will be used throughout the rest of the paper.

\begin{prop}
\label{prop:angles-at-critical-pts} Let $f$ be a Morse function
and $\Omega$ one of its Neumann domains.
\begin{enumerate}
\item \label{enu:prop-angles-at-critical-pts-1} If $\boldsymbol{c}$ is
a saddle point of $f$, then $\deg(\bs c)=4$ and the angle between
any two adjacent Neumann lines which meet at $\boldsymbol{c}$ is
$\tfrac{\pi}{2}$.
\item \label{enu:prop-angles-at-critical-pts-2} If $\boldsymbol{c}$ is
an extremal point of $f$ whose Hessian is not proportional to the metric $g$, then
 any two Neumann lines meet at $\boldsymbol{c}$ with
angle $0$, $\tfrac{\pi}{2}$ or $\pi$.
\item \label{enu:prop-angles-at-critical-pts-3} Let $\bs c$ be an intersection
point of a nodal line and a Neumann line. If $\bs c$ is a
saddle point, then the angle between those lines is $\tfrac{\pi}{4}$.
Otherwise, the angle is $\tfrac{\pi}{2}$.
\end{enumerate}
\end{prop}

\begin{rem}
More generally, if $\boldsymbol{c}$ is a saddle point and there exist coordinates $(x,y)$ near $\boldsymbol{c}$ in which $f$ is given by the homogeneous harmonic polynomial $\operatorname{Re} (x+iy)^k$, then $\deg(\bs c)=2k$. For a non-degenerate saddle the existence of such coordinates (with $k=2$) is an immediate consequence of the Morse lemma, so we obtain Proposition \ref{prop:angles-at-critical-pts}\eqref{enu:prop-angles-at-critical-pts-1} as a special case of this remark. Sufficient conditions for $f$ to be written in this form are given in \cite[Lem. 2.4]{C76}.
\end{rem}

The first and third parts of Proposition \ref{prop:angles-at-critical-pts}
were proved in \cite{fulling:14}, \cite[Thm.~4.2]{Ban:2004} and
\cite[Prop.~4.1]{abbe:2018}. The second part of the proposition
is proved below (see Remark~\ref{rem:proof} after the proof), using the following version of Hartman's theorem, which will also be used in the proofs of 
Lemma \ref{c1neumannline} and Proposition \ref{extens} to give a canonical description of the boundary 
of a Neumann domain near a cusp point.
\begin{prop}
\cite{Hartman:1960}
\label{prop: Hartman} Let $E$ be
an open neighbourhood of $\bs p\in\rz^{2}$. Suppose $F\in C^{2}(E,\rz^{2})$,
and let $\varphi$ be the flow of
the nonlinear system $\partial_{t}\varphi(t, x)=F(\varphi(t,x))$.
Assume that $F(\bs p)=\bs 0$ and the Jacobian $DF(\bs p)$ is
diagonalizable and its eigenvalues have non-zero real part. Then,
there exists a $C^{1}$-diffeomorphism $\Phi:U\rightarrow V$ of an open
neighbourhood $U$ of $\bs p$ onto an open neighbourhood $V$ of
the origin, such that $D\Phi(\bs p)=\Id$ and for each $ x\in U$ the flow line through
$\bs p$ is mapped by $\Phi$ to
\begin{equation}
\label{bettermap}
	\Phi(\varphi(t, x)) = \rme^{DF(\bs p)t}\Phi(x)
\end{equation}
for small enough $t$ values.
\end{prop}

\begin{rem}
The textbook version of Hartman's theorem in $n$ dimensions (see, for instance, \cite[p. 120]{Perko:2001}) only guarantees the existence of a homeomorphism $\Phi$. For $n=2$, the proposition above guarantees that $\Phi$ is a $C^1$-diffeomorphism, but for $n > 2$ further assumptions on the Jacobian are required to obtain this additional regularity. For instance, it suffices to assume that all of the eigenvalues of $DF(\bs p)$ are in the same (left or right) half plane; see \cite[p. 127]{Perko:2001}. That version of the theorem would be sufficient for our purposes, since we only apply Proposition \ref{prop: Hartman} at non-degenerate extrema, where all eigenvalues have the same sign. However, it is interesting to note that Proposition \ref{prop: Hartman} also applies at saddle points in two dimensions.
\end{rem}

\begin{proof}[Proof of Proposition \ref{prop:angles-at-critical-pts}\eqref{enu:prop-angles-at-critical-pts-2}]
Let $\boldsymbol{c}$ be an extremal point of $f$ whose Hessian
is not proportional to $g$. Since $\Hess f(\bs c)$ is non-degenerate, 
both eigenvalues
of $\Hess f(\bs c)$ are either strictly positive or strictly negative. We choose normal coordinates in an open neighbourhood $\widetilde{E}$ of $\bs c$,
with respect to which $\widetilde{E}$ is represented by an open subset $E\subset\R^{2}$, $\bs c$ corresponds to the origin $\bs{0} \in \R^2$, and $g_{ij}(\bs 0)=\delta_{ij}$.

We now apply Proposition \ref{prop: Hartman} to $F=-\grad f$. Since $DF(\bs 0)=-\Hess f(\bs 0)$ is diagonalizable
and has nonzero eigenvalues,  there exist $U\subset E$ and $V\subset\R^{2}$,
both containing the origin, and a $C^1$-diffeomorphism $\Phi:U\rightarrow V$, such that the
gradient flow lines are mapped by $\Phi$ to the flow lines $\rme^{-t\Hess f(\bs 0)}\Phi(x)$ of the linearized system. In \cite[Thm.~3.1]{fulling:14},
\cite[Prop.~4.1]{abbe:2018} it was shown that the angle between such 
flow lines at an extremal point is either $0$, $\tfrac{\pi}{2}$ or $\pi$, under the assumption 
that $\Hess f(\bs 0)$ is not a scalar matrix.
This assumption holds, as the Hessian is not proportional to the metric and we have chosen coordinates with respect to
which $g(\bs 0)$ is the identity.

It is left to relate the meeting angle between the gradient flow lines in $M$ and the corresponding  flow lines $\rme^{-t\Hess f(\bs 0)}\Phi(x)$ in $V$.  Since the tangent map $D\Phi(\bs 0) : T_{\bs 0} U \to T_{\bs 0} V$ is the identity, and $g_{ij}(\bs 0) = \delta_{ij}$, the meeting angle of any two curves at $\bs 0$ is preserved by
$\Phi$, hence this angle is either $0$, $\tfrac{\pi}{2}$ or $\pi$. This completes the proof.
\end{proof}

\begin{rem}
\label{rem:proof}
The argument for Proposition \ref{prop:angles-at-critical-pts}\eqref{enu:prop-angles-at-critical-pts-2} given in \cite[Prop.~4.1]{abbe:2018} is incomplete
and hence we have supplied a complete proof here.
In particular, the Taylor expansion argument used in the
proofs of \cite[Thm.~3.1]{fulling:14} and \cite[Prop.~4.1]{abbe:2018}
does not suffice here.
Substituting the Taylor expansion of $f$ in \eqref{eq:flow-1}
gives
\begin{equation}
\begin{pmatrix}x'(t)\\
y'(t)
\end{pmatrix}=-\Hess f(\bs c)\cdot\begin{pmatrix}x(t)\\
y(t)
\end{pmatrix}+\mathcal{O}\left(\left\Vert \left(x(t),y(t) \right)\right\Vert _{\R^{2}}^{2}\right), \label{eq: Taylor_of_flow_ODE}
\end{equation}
but this does not allow us to conclude that the flow may be approximated
by
\[
\begin{pmatrix}x(t)\\
y(t)
\end{pmatrix}\approx\rme^{-t\Hess f(\bs c)}\cdot\begin{pmatrix}x(0)\\
y(0)
\end{pmatrix},
\]
due to the possible coupling of higher order terms in (\ref{eq: Taylor_of_flow_ODE}).
\end{rem}

From Proposition \ref{prop:angles-at-critical-pts}\eqref{enu:prop-angles-at-critical-pts-2}
we see that the boundary of a Neumann domain may possess a cusp (when the meeting angle is $0$)
and so it can fail to be Lipschitz continuous. Furthermore, it may
even fail to be of class $C$. Recall that a boundary of a domain
is said to be of class $C$ if it can be locally represented as the
graph of a continuous function; alternatively, if the domain has the
segment property (see \cite{Evans:1987} 
or \cite{MP01} for details). To demonstrate that this is a subtle
property, we bring as an example the domains
\begin{equation}
\begin{split}
	\Omega_1 &= \big\{(x,y)\in\mathbb{R}^{2}: \tfrac12 x^2 < y<x^{2}, \ 0 < x < 1\big\},\label{nonCexample} \\
	\Omega_2 &= \big\{(x,y)\in\mathbb{R}^{2}:- x^2<y<x^{2}, \ 0 < x < 1 \big\},
\end{split}
\end{equation}
which are shown in Figure \ref{fig:nonC}. The domain $\Omega_1$ does not satisfy
the segment property at the origin, and hence is not of class $C$,
even though its boundary is the union of two smooth curves. On the other hand, $\Omega_2$ (which contains $\Omega_1$) \emph{is} of class $C$. This example will be important later, in the proof of Proposition~\ref{extens}.

\begin{figure}[ht]
\begin{center}
\begin{tikzpicture}
\begin{axis}[xmin=-0.05,ymin=-1.05,xmax=1.05,ymax=1.05, samples=50, yticklabels=\empty, xticklabels=\empty]
  \addplot[domain=0:1, very thick] (x,x*x);
  \addplot[domain=0:1, very thick] (x,0.5*x*x);
    \addplot[domain=0.5:1, very thick] (1,x);
\end{axis}
\end{tikzpicture}
\begin{tikzpicture}
\begin{axis}[xmin=-0.05, ymin=-1.05, xmax=1.05, ymax=1.05, samples=50, yticklabels=\empty, xticklabels=\empty]
  \addplot[domain=0:1, very thick] (x,x*x);
  \addplot[domain=0:1, very thick] (x,-x*x);
  \addplot[domain=-1:1, very thick] (1,x);
\end{axis}
\end{tikzpicture}

\caption{The regions $\Omega_1$ (left) and $\Omega_2$ (right) defined in \eqref{nonCexample} both have a cusp at the origin. However, $\Omega_1$ is not of class $C$, whereas $\Omega_2$ is.}
\label{fig:nonC}
\end{center}
\end{figure}
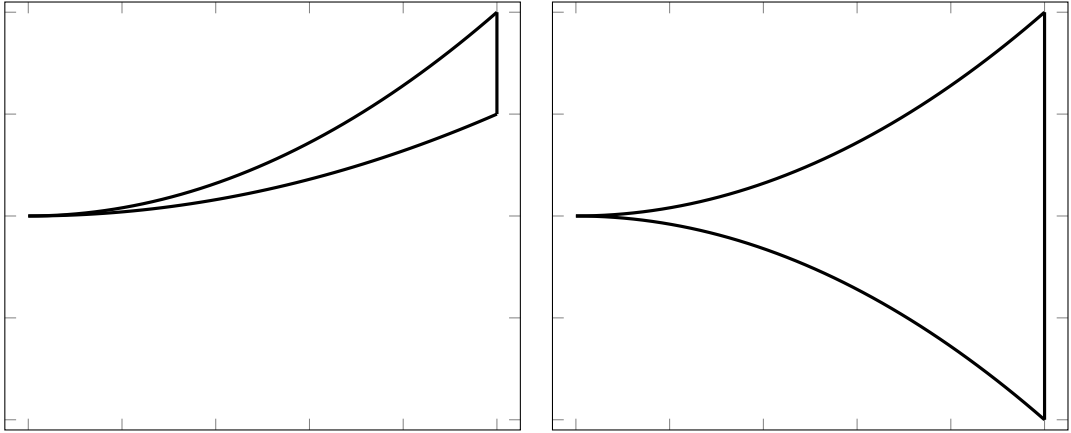

We add that there is very little known in general regarding the asymptotic
behavior of Neumann lines near cusps. In particular, methods to treat
cusps in a spectral theoretic context as in, e.g., \cite{simon:92,flamencour:20,band2017ground}
have to be generalized for our purpose.

We end this section by examining Theorem \ref{thm:topological-properties-manifolds} and its 
implications for the structure of Neumann domains. By the statement of the
theorem, the boundary of a Neumann domain always contains a maximum
and a minimum, but no other extrema.  It follows that each Neumann domain 
must belong to one of following two types (illustrated in Figure \ref{fig:Neumann-domains-schematic}):
\begin{itemize}
\item a \emph{regular Neumann domain} has on its boundary a maximum and a minimum,
each of degree at least two (see Definition \ref{def: Neumann_line_and_degree_of_critical_pt});
\item a \emph{cracked Neumann domain} has on its boundary an extremal point which
is of degree one.
\end{itemize}

Moreover, since the boundary is made up of Neumann 
lines, it must contain at least one saddle point. If $f$ is Morse--Smale the boundary contains at most two saddle points, by Theorem \ref{thm:topological-properties-manifolds}\eqref{enu:thm-topological-properties-manifolds-4}, but for a general Morse function it is possible to have more. The possible existence of additional saddle points is irrelevant for our analysis, however, since the boundary is Lipschitz near these points by Proposition \ref{prop:angles-at-critical-pts}\eqref{enu:prop-angles-at-critical-pts-1}.

Numerical observations suggest that generic Neumann domains are
regular. However, it is not hard to construct Morse functions having cracked Neumann domains; see Appendix~\ref{app:typeii}. Theorems \ref{thm:discrete} and \ref{thm:restriction} apply to
both types of domains, but in the proofs we need to pay careful attention
to cracked domains. In particular, a cracked Neumann domain is not of class $C$, as the domain lies on both sides of its boundary.

\begin{figure}[ht]
\centering{}\includegraphics[width=0.8\textwidth]{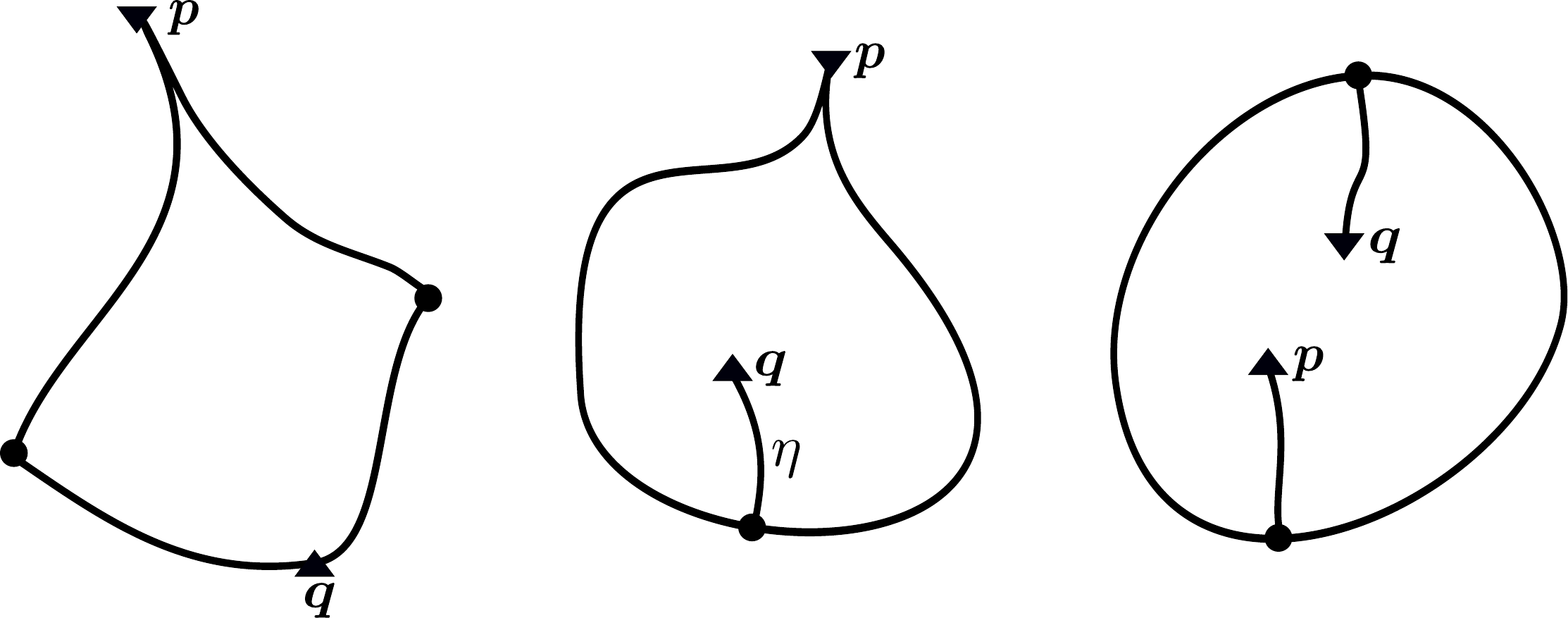} \caption{Possible types of Neumann domains for a Morse function: regular (left); cracked (center); and doubly-cracked (right). Saddle points are represented by balls, maxima by triangles pointing upwards and vice versa for minima. If $f$ is Morse--Smale, its Neumann domains must look like one of the first two examples, with either one or two saddle points on the boundary.
For the cracked domain shown in the center,
$\eta$ is the only Neumann line connected to $\protect\bs q$,
hence $\deg(\protect\bs q)=1$. If $f$ is not Morse--Smale, its Neumann domains can have additional saddle points on the boundary, and can have both extremal points of degree one, as shown on the right. (This last example has a Neumann line connecting two saddle points, which is not possible if $f$ is Morse--Smale, by Lemma \ref{prop:Morse-Smale-on-2d}.)
}
\label{fig:Neumann-domains-schematic}
\end{figure}

\begin{rem}
\label{rem:notC}
In summary, a Neumann domain may fail to be of class $C$ for two reasons: 1) a cusp on the boundary; or 2)  a crack in the domain, i.e. a Neumann line contained in the interior of $\overline\Omega$. These are the main technical obstacles to be  overcome in proving Theorems \ref{thm:discrete} and \ref{thm:restriction}.
\end{rem}

\section{Sobolev spaces on Neumann domains}
\label{sec:Sobolev}
We now discuss properties of Sobolev spaces on Neumann domains. As described in the introduction, and indicated in Proposition \ref{prop:angles-at-critical-pts}\eqref{enu:prop-angles-at-critical-pts-2} (see also Remark \ref{rem:notC}), the difficulty is that the boundary of a Neumann domain need not be of class $C$, so standard density and compactness results do not apply.

In Section \ref{sec:Sobolevdef} we define Sobolev spaces on a Neumann domain and various subsets of its boundary.
In Sections \ref{secdisec} and \ref{sec:truncate} we describe some technical constructions (dissection and truncation) that allow us to deal with cracks and cusps. Finally, in Section \ref{sec:dense} we prove the main result of this section, Proposition \ref{extens}, which establishes density and embedding properties for the space $W^{1,2}(\Omega)$ on a Neumann domain.

\subsection{Preliminaries}
\label{sec:Sobolevdef}
As above, we assume that $(M,g)$ is a smooth, closed, connected, oriented Riemannian surface. For an open submanifold $N\subset M$, the Sobolev space $W^{k,2}(N)$ is defined to be the completion of $C^{\infty}(N)$ with respect to the norm
\be
\label{snorm}
\|f\|^2_{W^{k,2}(N)}:=\sum\limits_{j=0}^k\int\limits_{N}|\nabla^j f|^2 ,
\ee
where $\nabla$ denotes the covariant derivative with respect to the metric $g$. The norm depends on $g$, but since $M$ is compact, different metrics will produce equivalent norms. We will sometimes take advantage of this fact and compute the Sobolev norm using a metric $\tilde{g}$ defined in a local coordinate chart to have components $\tilde{g}_{ij} = \delta_{ij}$ (so that covariant derivatives become partial derivatives, the Riemannian volume form reduces to the Euclidean one, etc.). This allows us to apply standard methods in the theory of Sobolev spaces on Lipschitz domains in $M$.

Now suppose that $N\subset M$ is an open submanifold with Lipschitz boundary. We will later choose $N$ to be a Neumann domain $\Omega$, or a proper subset thereof (see Section \ref{sec:truncate}) if $\partial\Omega$ has a crack or a cusp. We define the boundary Sobolev spaces $H^s(\partial N)$ for $|s| \leq 1$ via the Fourier transform and a suitable partition of unity, following \cite[p.~96]{McL}, so that the dual space is given by $H^s(\partial N)^* = H^{-s}(\partial N)$. Moreover, for any open subset $\Gamma \subset \partial N$ we let
\be
\label{sobolev}
\ba
	H^s(\Gamma) &:= \big\{ f\big|_\Gamma : f \in H^s(\partial N) \big\}, \\
	\widetilde H^s(\Gamma) &:= \text{closure of $C^\infty_0(\Gamma)$ in $H^s(\Gamma)$}.
\ea
\ee
The space $\widetilde H^s(\Gamma)$ has an equivalent description that is often useful in practice:
\be
\label{sobolev2}
		\widetilde H^s(\Gamma) = \big\{ f \in H^s(\partial N) : \supp\,f \subset \overline\Gamma \big\}.
\ee
This equivalence follows from \cite[Thm. 3.29]{McL}. Another convenient description, valid for $s \geq 0$, is
\be
\label{sobolev3}
	\widetilde H^s(\Gamma) = \big\{ f \in L^2(\Gamma) : \tilde f \in H^s(\partial N) \big\},
\ee
where $\tilde f$ denotes the extension of $f$ by zero to $\partial N \setminus \Gamma$; this is \cite[Thm. 3.33]{McL}.

It follows from the definitions that $\widetilde H^s(\Gamma) \subset H^s(\Gamma)$ for all $|s| \leq 1$, and it is well known that these spaces coincide for $|s| < 1/2$. However, for $|s| \geq 1/2$ we have $\widetilde H^s(\Gamma) \subsetneq H^s(\Gamma)$ whenever $\Gamma$ is a proper subset of $\partial N$. To see this, consider the constant function $f \equiv 1$ on $\Gamma$, which is clearly in $H^s(\Gamma)$ for any $s$. It is easily verified that its extension $\tilde f$, which is just the indicator function $\chi_{_\Gamma}$, is not in $H^s(\partial N)$ for $s \geq 1/2$, in which case we conclude from \eqref{sobolev3} that $f \notin \widetilde H^s(\Gamma)$. This distinction between the $H^s$ and $\widetilde H^s$ spaces will be important when we consider the normal derivative of a function restricted to a subset of the boundary; see in particular Lemma \ref{lem:split} and its application in the proof of Proposition \ref{prop:domain_of_Laplacian}.

The $\widetilde H^s$ spaces arise naturally as duals to the $H^s$ spaces. That is, for any $|s| \leq 1$ we have $H^s(\Gamma)^* = \widetilde H^{-s}(\Gamma)$, from \cite[Thm.~3.30]{McL}. In particular,
\be
\label{nest}
	\widetilde H^{-\frac12}(\Gamma) = H^{\frac12}(\Gamma)^* \subsetneq \widetilde H^{\frac12}(\Gamma)^* = H^{-\frac12}(\Gamma).
\ee
Using \eqref{sobolev} we obtain
\be
\label{equality}
\ba
	\ell=0 \quad \text{in} \quad H^{-s}(\Gamma) 
&\quad \Longleftrightarrow \quad \ell(f)=0 \quad \forall f \in C^{\infty}_0(\Gamma).
\ea
\ee
We thus define for $0 \leq s \leq 1$ the mapping 
\be
\label{wast}
\ba
&\cdot^{\dual}:L^2(\Gamma)\rightarrow H^{s}(\Gamma)^* , \\
&\quad g^{\dual}(f):=\langle f, g\rangle_{L^2(\Gamma)}, \quad f \in H^s(\Gamma),
\ea
\ee
observing that the $L^2$ inner product is well defined because $H^s(\Gamma) \subset L^2(\Gamma)$ for $0 \leq s \leq 1$. As a result, we will often abuse notation and use integral notation to denote the action of $\ell \in H^s(\Gamma)^*$ on $f \in H^s(\Gamma)$, i.e. we will write 
\[
	\ell(f) = \int_\Gamma \ell f
\]
even when $\ell$ is not in the range of the map $\cdot^{\dual}$; see in particular Green's identity \eqref{ggsmooth} below.

Given a decomposition $\partial N = \overline \Gamma_1 \cup \overline{\Gamma}_2$, where $\Gamma_1$ and $\Gamma_2$ are disjoint, open subsets of $\partial N$, and a distribution $\ell \in H^{-s}(\partial N)$ for some $s \geq 0$, we have $\ell|_{\Gamma_i} \in H^{-s}(\Gamma_i)$ for $i=1,2$. For $0 \leq s < 1/2$ we obtain the decomposition
\[
	\ell(\phi) = \ell|_{\Gamma_1}\big(\phi|_{\Gamma_1}\big) + \ell|_{\Gamma_2}\big(\phi|_{\Gamma_2}\big)
\]
for every $\phi \in H^s(\partial N)$. However, this is not true for $s \geq 1/2$, and in fact the right-hand side is not even defined in this case, since $\phi|_{\Gamma_i} \in H^s(\Gamma_i)$, whereas $\ell|_{\Gamma_i} \in H^{-s}(\Gamma_i)$ might not be contained in $H^s(\Gamma_i)^*$, as indicated in \eqref{nest}. However, such a splitting does hold for $\ell$ if we assume that $\Gamma_1$ and $\Gamma_2$ are separated by a third subset $\Gamma_0$ on which $\ell$ vanishes.

\begin{lem}
\label{lem:split}
Suppose $\partial N = \overline \Gamma_0 \cup \overline \Gamma_1 \cup \overline{\Gamma}_2$, where $\Gamma_0, \Gamma_1, \Gamma_2$ are disjoint, open subsets of $\partial N$ with $\overline \Gamma_1 \cap \overline{\Gamma}_2 = \emptyset$. If $\ell \in H^{-\frac12}(\partial N)$ vanishes on $\Gamma_0$, then $\ell|_{\Gamma_i} \in \widetilde H^{-\frac12}(\Gamma_i)$ for $i=1,2$, and
\be
	\ell(\phi) =  \ell|_{\Gamma_1}\big(\phi|_{\Gamma_1}\big) + \ell|_{\Gamma_2}\big(\phi|_{\Gamma_2}\big)
\ee
for every $\phi \in H^{\frac12}(\partial N)$.
\end{lem}

Such a partition of the boundary is illustrated in Figure \ref{fig:Neumann-domains-dissection-truncation}, where $N = \Omega_t \cap \Omega_{\rm r}$, $\overline\Gamma_0 = \gamma_{0,t}$, $\overline\Gamma_1 = \tilde\eta$ and $\overline\Gamma_2 = \gamma_{-,t}$.

\begin{proof}
We will use \eqref{sobolev2} to prove that $\ell|_{\Gamma_1} \in \widetilde H^{-\frac12}(\Gamma_1)$. This does not follow immediately, however, since $\ell$ is not necessarily supported in $\overline\Gamma_1$. Therefore, we will create a modified distribution, $\ell_1$, such that $\supp \ell_1 \subset \overline\Gamma_1$ and $\ell|_{\Gamma_1} = \ell_1|_{\Gamma_1}$.

Since $\overline \Gamma_1 \cap \overline{\Gamma}_2 = \emptyset$, we can find a smooth bump function $\chi_1$ that equals $1$ on $\Gamma_1$ and $0$ on $\Gamma_2$. Consider the distribution $\ell_1(\phi) := \ell(\chi_1 \phi)$, which is in $H^{-\frac12}(\partial N)$. If $\supp \phi \subset \overline \Gamma_0 \cup \Gamma_2$, then $\supp (\chi_1 \phi) \subset \Gamma_0$, and hence $\ell_1(\phi) = \ell(\chi_1 \phi) = 0$, because $\ell$ vanishes on $\Gamma_0$. This shows that $\supp \ell_1 \subset \partial N \setminus (\overline \Gamma_0 \cup \Gamma_2) = \overline \Gamma_1$. On the other hand, if $\supp\phi \subset \Gamma_1$, then $\chi_1 \phi = \phi$, and hence $\ell_1(\phi) = \ell(\phi)$. We have thus shown that $\ell|_{\Gamma_1} = \ell_1|_{\Gamma_1} \in \widetilde H^{-\frac12}(\Gamma_1)$.

Similarly, we obtain $\ell_2 \in H^{-\frac12}(\partial N)$ with $\supp \ell_2 \subset \overline\Gamma_2$ and $\ell|_{\Gamma_2} = \ell_2|_{\Gamma_2} \in \widetilde H^{-\frac12}(\Gamma_2)$. It follows that the distribution
\[
	\hat \ell := \ell - \ell_1 - \ell_2 \in H^{-\frac12}(\partial N)
\]
has support in $\partial N \setminus (\Gamma_0 \cup \Gamma_1 \cup \Gamma_2) = (\overline \Gamma_0 \cap \overline{\Gamma}_1) \cup (\overline \Gamma_0 \cap \overline{\Gamma}_2)$, which is a finite set. However, a distribution in $H^{-\frac12}$ cannot be supported on a finite set of points, by \cite[Lem.~3.39]{McL}, so we conclude that $\hat \ell$ is identically zero, which completes the proof.
\end{proof}

We next introduce the weak Laplace--Beltrami operator, $\Delta \colon W^{1,2}(\Omega) \to W^{1,2}_0(\Omega)^*$, where $\Omega \subset M$ is any open subset of $M$. By definition, $\Delta \psi = f$ means
\begin{equation}
\label{weakBT}
\int_\Omega \left<\grad \psi, \grad \phi\right> = \int_\Omega f \phi
\end{equation}
for all $\phi \in W^{1,2}_0(\Omega)$, where the integral on the right-hand side is shorthand for the action of $f \in W^{1,2}_0(\Omega)^*$ on $\phi \in W^{1,2}_0(\Omega)$. If $\Delta \psi = f \in L^2(\Omega)$, then this is a genuine $L^2$ inner product of $f$ and $\phi$.

For a Lipschitz domain $N$ the trace map $\cdot|_{\partial N}:W^{1,2}(N)\rightarrow H^{\frac{1}{2}}(\partial N)$ is continuous. The weak version of Green's identity says that for any $\psi\in W^{1,2}(N)$ with $\Delta \psi\in L^2(N)$, there exists a unique $\partial_{\nu}\psi\in H^{-\frac{1}{2}}(\partial{N})$ such that 
\begin{equation}
\label{ggsmooth}
	\int_N \left<\grad \psi, \grad \phi \right> = \int_N (\Delta \psi) \phi + \int_{\partial N} \phi (\partial_\nu \psi)
\end{equation}
for all $\phi\in W^{1,2}(N)$, \cite[Thm.~4.4]{McL}.  The boundary term has to be understood as the action of $\partial_{\nu}\psi\in H^{-\frac12}(\partial{N})$ on $\phi|_{\partial{N}}\in H^{\frac12}(\partial{N})$, i.e., $(\partial_{\nu}\psi)(\phi|_{\partial N})$, but to simplify the presentation we use the integral notation of \eqref{ggsmooth}. 

Finally, consider an open subset $\Gamma \subset \partial N$. For a function $\psi \in W^{1,2}(N)$ we define $\psi\big|_\Gamma$ to be the restriction of $\psi\big|_{\partial N} \in H^{1/2}(\partial N)$ to $\Gamma$, so that \eqref{sobolev} implies $\psi\big|_\Gamma \in H^{1/2}(\Gamma)$. Similarly, if $\psi \in W^{1,2}(N)$ and $\Delta \psi \in L^2(N)$, we have $\partial_\nu \psi \big|_\Gamma \in H^{-1/2}(\Gamma)$. It is not necessarily true that $\partial_\nu \psi \big|_\Gamma$ is contained in the smaller space $\widetilde H^{-1/2}(\Gamma) = H^{1/2}(\Gamma)^*$; c.f. \eqref{nest}. However, this will be the case if $\partial_\nu \psi$ vanishes on $\partial N \setminus \overline\Gamma$, by Lemma \ref{lem:split}. This fact will play a crucial role in the proof of Proposition \ref{prop:domain_of_Laplacian} below.

We conclude the section by explaining our decision to use $W^{k,2}$-Sobolev spaces on $N$ but $H^s$-Sobolev spaces on $\partial N$. Recall that $H^1(\Omega) \subset W^{1,2}(\Omega)$ holds for any open set $\Omega$, but the inclusion can be strict unless one has additional regularity of the boundary. In Definition \ref{def:NeuLap} we construct the Neumann Laplacian as the self-adjoint operator corresponding to a non-negative, symmetric bilinear form. For this we require the form to be closed, which is the case if the form domain is $W^{1,2}(\Omega)$, but need not be true if the form domain is $H^1(\Omega)$.

On the other hand, the $H^s$-Sobolev spaces, defined via the Fourier transform, provide a more natural setting for the discussion of traces: If $N$ is an open submanifold with Lipschitz boundary, the trace map $\cdot|_{\partial N} \colon C^1(\bar N) \to C^1(\partial N)$ has a bounded, surjective extension $\cdot|_{\partial N}:H^1(N)\rightarrow H^{\frac{1}{2}}(\partial N)$. For $N$ Lipschitz we have the equality $H^1(N) = W^{1,2}(N)$, and hence a well-defined trace map $\cdot|_{\partial N}:W^{1,2}(N)\rightarrow H^{\frac{1}{2}}(\partial N)$.

\subsection{Dissections of Neumann domains}
\label{secdisec}
The boundary of a cracked Neumann domain cannot be of class $C$, whether or not there is a cusp on the boundary, due to the Neumann line $\eta$ contained in the interior of $\overline\Omega$; see Figure \ref{fig:Neumann-domains-schematic}. We deal with this by dissecting such a Neumann domain into two pieces, as shown in Figure \ref{fig:Neumann-domains-dissection}, where one piece has Lipschitz boundary, and the other has boundary that is Lipschitz except possibly at a cusp point, i.e. it has the same regularity as a regular Neumann domain. For doubly-cracked domains as in Figure \ref{fig:Neumann-domains-schematic}, an analogous statement holds as the proof for that case is essentially the same. The dissection thus reduces many of the proofs for cracked domains to the corresponding results for regular domains.

This dissection is made possible by the following lemma.
\begin{lem}
\label{c1neumannline}
Assume $f$ is a Morse function, and let $\gamma$ be a Neumann line. Then $\gamma$ has finite length $L(\gamma)<\infty$, and admits an arc-length parametrization with $\gamma\in C^1\big([0,L(\gamma)]\big)$, i.e., boundary points are included. 
\end{lem}
\begin{proof}
We decompose $\gamma=\gamma_0\cup\gamma_1 \cup \gamma_2$, where $\gamma_1$ is defined in a small neighbourhood of the initial endpoint of $\gamma$ and $\gamma_2$ is defined in a small neighbourhood of the terminal endpoint. Then it is enough to prove the corresponding statement for $\gamma_0$, $\gamma_1$ and $\gamma_2$.

The result for $\gamma_0$ follows by standard results for flows of smooth vector fields. Definition~\ref{def: Neumann_line_and_degree_of_critical_pt} implies that the endpoints of $\gamma$ are critical points of $f$. If the initial endpoint (which we label $\bs c$) is a saddle, then the result for $\gamma_1$ follows, e.g., by \cite[Thm.~4.2~(p.~94)]{Ban:2004}. On the other hand, if $\bs c$ is an  extremum we use the map $\Phi$ from Proposition \ref{prop: Hartman}. Then $\Phi \circ\gamma_1$ is a flow line generated by $\rme^{-\Hess f(\bs{c})t}$, and hence satisfies the properties of the claim, i.e. it is $C^1$ up to the endpoint and has finite length. As $\Phi^{-1}$ is a $C^1$ map and $\gamma_1 = \Phi^{-1}\circ (\Phi\circ\gamma_1)$ is a composition of $C^1$ functions, the claim for $\gamma_1$ follows. The proof for $\gamma_2$ is identical.
\end{proof}

\begin{figure}[ht]
\centering{}\includegraphics[width=0.6\textwidth]{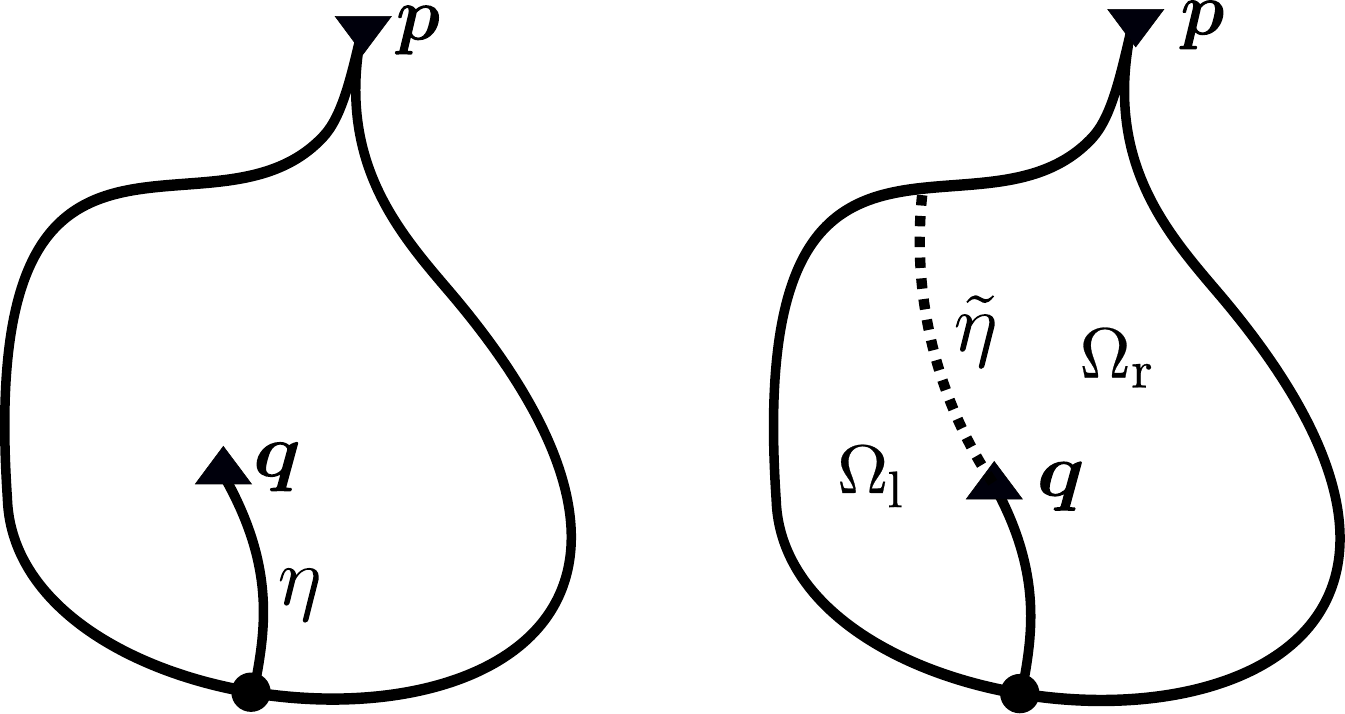} \caption{The dissection of a cracked Neumann domain, as given in \eqref{dissect}. The Neumann line $\eta$ 
is extended to a Lipschitz curve $\eta \cup \tilde\eta$, so that
$\Omega_{\mathrm{l}}$ is a Lipschitz domain and $\Omega_{\mathrm{r}}$
possesses a cusp at $\bs{p}$.
}
\label{fig:Neumann-domains-dissection}
\end{figure}

Now suppose that $\Omega$ is a cracked Neumann domain. The doubly-cracked case in Figure \ref{fig:Neumann-domains-schematic} can be treated analogously. Denote by $\bs{q}$ the extremum in the interior of $\overline{\Omega}$, and let $\eta$ be the Neumann line attached to $\bs q$;  see Figure \ref{fig:Neumann-domains-dissection}. Using Lemma~\ref{c1neumannline} we may dissect $\Omega$ into two parts ($\Omega_{\mathrm{l}}$ and $\Omega_{\mathrm{r}}$ in Figure \ref{fig:Neumann-domains-dissection}) so that
\be
\label{dissect}
	\Omega \setminus \tilde\eta =\Omega_{\mathrm{l}} \cupdot \Omega_{\mathrm{r}},
\ee
where $\tilde{\eta}$ is a Lipschitz curve in $\Omega$ joining $\bs{q}$ with a non-cusp point of $\partial\Omega$, hence $\eta \cup \tilde{\eta}$ is a Lipschitz curve by Lemma~\ref{c1neumannline}. This dissection induces an isometric embedding
\be
\label{isoemb}
\ba
	W^{1,2}(\Omega) &\rightarrow W^{1,2}(\Omega_{\mathrm{l}})\oplus W^{1,2}(\Omega_{\mathrm{r}}) \\
	\phi &\mapsto \big( \phi|_{\Omega_\mathrm{l}}, \phi|_{\Omega_\mathrm{r}} \big).
\ea
\ee

\subsection{Truncated Neumann domains}
\label{sec:truncate}
To deal with potential cusps at the maximum and minimum of $f$, we introduce truncated versions of $\Omega$. Denoting by $\bs{p}\in\Min(f)$ and $\bs{q}\in\Max(f)$ the minimum and maximum of $f$ in $\overline{\Omega}$, we observe that $f(\bs q) < f(\bs p)$, since otherwise $f$ would be constant on $\overline\Omega$, which is not possible as it is a Morse function. Adding a constant to $f$, which does not affect the gradient flow lines, we can thus assume that $f(\bs q) < 0 < f(\bs p)$. (In the special case that $f$ is an eigenfunction this condition holds automatically, so it is not necessary to shift $f$.)

We then define for each $0 < t < 1$ the truncated domains
\be
\label{omegat}
\ba
&\Omega_{t}:=\begin{cases}
\{x\in\Omega: \ f(x) <tf(\bs{q})\}, & \bs{q} \ \text{is a cusp}, \ \bs{p} \ \text{is not},\\
\{ x\in\Omega: \ tf(\bs{p}) < f(x) \}, & \bs{p} \ \text{is a cusp},\ \bs{q} \ \text{is not},\\
\{x\in\Omega: \ tf(\bs{p})<f(x) <tf(\bs{q})\}, & \bs{q} \ \text{and} \ \bs{p} \ \text{are cusps},\\
\Omega, & \text{otherwise}.
\end{cases}
\ea
\ee
Some examples of this construction are shown in Figure \ref{fig:Neumann-domains-schematicaaa}.

The boundary of $\Omega_t$ can be decomposed as $\partial\Omega_t=\gamma_{\pm,t}\cup\gamma_{0,t}$, where $\gamma_{\pm,t}$ are level lines defined by 
\be
\label{gammapmdef}
	\gamma_{+,t} = \big\{x : f(x) = t f(\bs q) \big\}, \qquad 	\gamma_{-,t} = \big\{x : f(x) = t f(\bs p) \big\},
\ee
and $\gamma_{0,t} = \pO_t  \cap \pO$ is the part of $\pO$ that remains after the truncation. Note that $\gamma_{0,t}\neq\emptyset$, and Proposition \ref{prop:angles-at-critical-pts}\eqref{enu:prop-angles-at-critical-pts-3} implies that $\gamma_{\pm,t}$ meets $\partial\Omega$ perpendicularly, except for a finite number of exceptional times where $\gamma_{\pm,t}$ meets $\partial\Omega$ at a saddle point, in which case the meeting angle is $\frac{\pi}{4}$; see Figure~\ref{fig:Neumann-domains-schematicaaa}.

For a truncated Neumann domain $\Omega_t$ we denote its complement in $\Omega$ by $\Omega_t^c := \Omega\setminus\Omega_t$. For any $0 < t < 1$ and sufficiently small $\epsilon>0$, we can find a smooth cutoff function $\chi$ on $M$ such that
\be
\label{chidef}
	\chi(x) =
\begin{cases}
0, & x\in \Omega_t,\\
1, & x\in\Omega_{t+\epsilon}^c.
\end{cases} 
\ee
If desired, we can assume that $\chi$ is of the form $\alpha \circ f$ for some $\alpha \in C^{\infty}(\R)$, in which case $\chi$ has the same level lines as $f$. For the arguments to follow, however, a generic smooth cutoff will suffice.
\begin{figure}[ht]
\centering{}\includegraphics[width=0.55\textwidth]{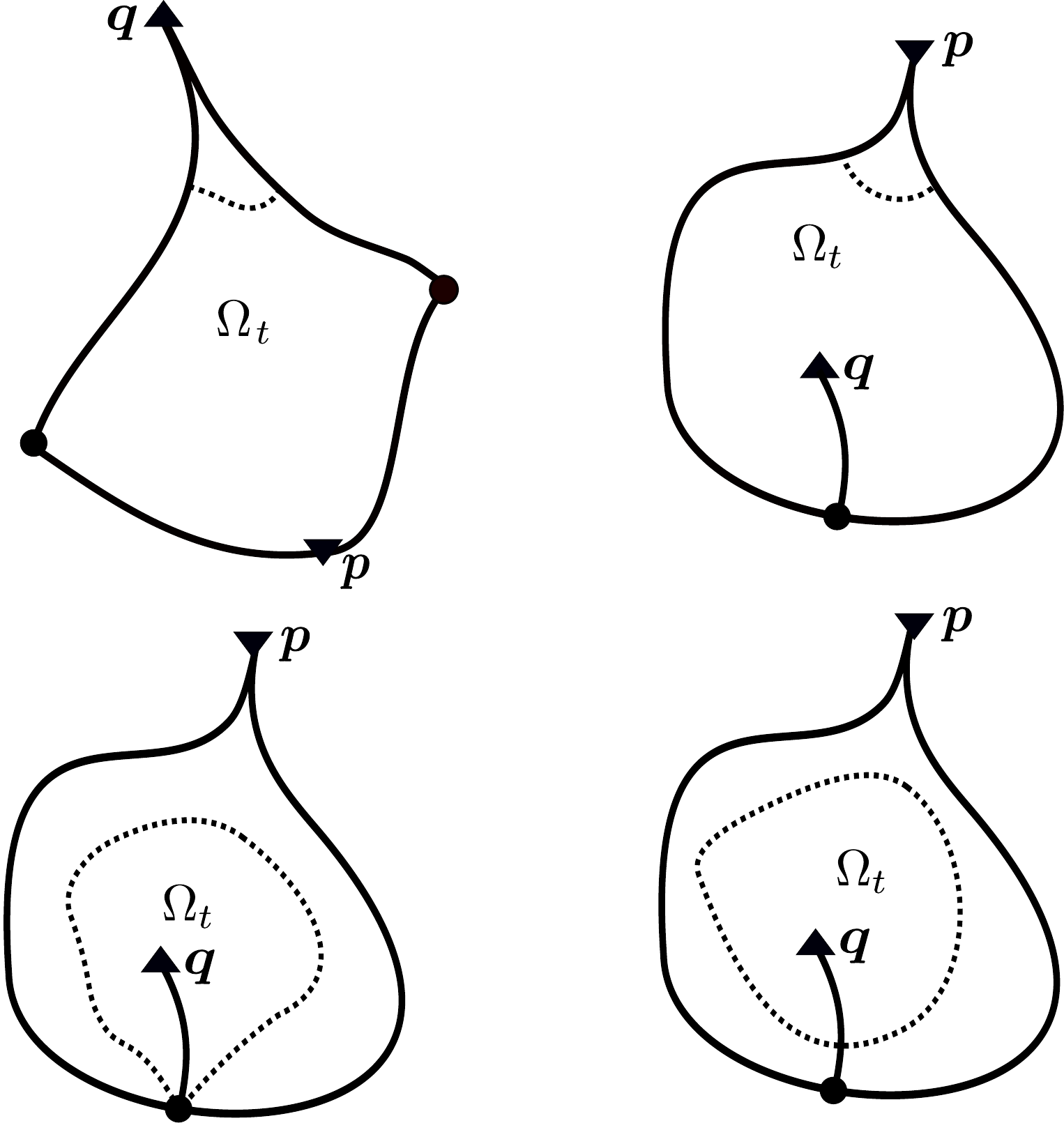} \caption{Neumann domains and their truncations, with the dotted line indicating the curve $\gamma_{\pm,t}$. The top two figures show regular and cracked domains for $t$ close to $1$. For the same cracked domain the bottom left figure shows an exceptional value of $t$, where $\gamma_{\pm,t}$ meets $\partial\Omega$ at angle $\frac{\pi}{4}$, and the bottom right figure shows a smaller value of $t$.}
\label{fig:Neumann-domains-schematicaaa}
\end{figure}
\subsection{Density and embedding results}
\label{sec:dense}
We now state and prove the main result of this section.
\begin{prop}
\label{extens}
Let $(M,g)$ be a closed, connected, oriented Riemannian surface. If $\Omega \subset M$ is a Neumann domain of a Morse function $f$, the following hold:
\begin{enumerate}
\item \label{compprime} the embedding $W^{1,2}(\Omega)\rightarrow L^2(\Omega)$ is compact;
\item \label{denseprime} if $\Omega$ is regular, then $C^1(\overline{\Omega})$ is dense in $W^{1,2}(\Omega)$;
\item \label{dense} if $\Omega$ is cracked, then there exists $t\in(0,1)$ such that the set of functions
\be
\label{tildec}
	\big\{ \phi\in W^{1,2}(\Omega): \ \phi|_{\overline{\Omega_t^c}} \in C^1(\overline{\Omega_t^c}) \big\}
\ee
is dense in $W^{1,2}(\Omega)$.
\end{enumerate}
\end{prop}

The result is known if $\partial\Omega$ is of class $C$ (see \cite{MP01}) but, as noted above, the boundary of a Neumann domain does not need to have this property. If $\Omega$ is cracked and $\phi \in W^{1,2}(\Omega)$, its values on opposite sides of the crack $\eta$ need not be close, so we cannot hope to approximate it by a function in $C^1(\overline{\Omega})$. However, by choosing $t$ sufficiently large, we can ensure that $\overline{\Omega_t^c}$ is disjoint from $\eta$, and hence \eqref{tildec} holds.

The idea of the proof is to use Hartman's theorem (Proposition \ref{prop: Hartman}) to find a canonical description of the boundary near a cusp, and then apply the following lemma, which allows us to extend functions to a larger domain which still has a cusp but is of class $C$; c.f. the domains $\Omega_1$ and $\Omega_2$ in Figure \ref{fig:nonC}.

\begin{lem}\cite[Lem.~5.4.1/1~(p.~285)]{MP01}
\label{lem:MP}
Consider the domain
\[
	\tilde\Omega = \big\{(x,y) \in \R^2 : c_1 \vartheta(x) < y < c_2 \vartheta(x), \ 0 < x < 1 \big\}
\]
for some $c_1 < c_2$, where $\vartheta \in C^{0,1}([0,1])$ is an increasing function with $\vartheta(0) = 0$ and $\vartheta'(t) \to 0$ as $t \to 0$, and define
\begin{equation}
	G = \big\{ (x,y) \in \R^2 : |y| < M \vartheta(x), \ 0 < x < 1 \big\}
\end{equation}
for $M \geq \max\{ |c_1|, |c_2|\}$.
Then there exists a continuous extension operator $\mc{E} \colon W^{1,2}(\tilde\Omega) \to W^{1,2}(G)$.
\end{lem}

We will apply this lemma with $\vartheta(x) = x^\alpha$ for some $\alpha > 1$.

\begin{proof}[Proof of Proposition \ref{extens}]
We first prove \eqref{compprime} and \eqref{denseprime} for regular Neumann domains. Only the behavior near the cusps has to be investigated, as they are the only possible non-Lipschitz points on $\partial\Omega$. A cusp is either a maxima or a minima by Proposition~\ref{prop:angles-at-critical-pts}. Without loss of generality, let $\bs{c}\in\Max(f)$ be the only cusp on $\partial{\Omega}$.

We localize at $\bs{c}$ by taking a smooth cutoff function $\chi$, as in \eqref{chidef}, that equals $1$ in $\Omega_{t+\epsilon}^c$ and vanishes in $\Omega_t$, and hence is supported in $\Omega_{\bs c} := \Omega_{t}^c$. 
Now take $\phi\in W^{1,2}(\Omega)$. We write $\phi=\chi\phi+(1-\chi)\phi$ and observe that $\chi\phi,\,(1-\chi)\phi\in W^{1,2}(\Omega)$. Thus, it is sufficient to prove the statements for both functions separately. For the latter function the observation that it is supported in a Lipschitz domain implies both \eqref{compprime} and \eqref{denseprime} in Proposition \ref{extens}. 

For the former space we choose $t$ close to $1$ and employ Proposition \ref{prop: Hartman}. Let $\Phi$ be the resulting $C^1$-diffeomorphism and define $\tilde{\Omega}_{\bs{c}} = \Phi(\Omega_{\bs c})$. Owing to \eqref{bettermap}, the image in $\tilde{\Omega}_{\bs{c}}$ of the two boundary curves meeting at $\bs{c}$ are flow lines obeying $\partial_t \gamma =-\Hess f(\bs{c})\gamma$. These are generated by $\rme^{-t\Hess f(\bs{c})}x_0$ where $x_0$ is a suitable point on $\gamma$. An easy calculation as in \cite[Section~3]{fulling:14} and \cite[Proof~of~Prop.~4.1]{abbe:2018} shows that the flow lines near the origin may be parametrized in suitable coordinates by $\gamma(x)=(x,c x^{\alpha})$, where $\alpha>0$ depends only on the eigenvalues of $\Hess f(\bs{c})$ (in fact only on their ratio). This implies that near the origin, the domain $\tilde{\Omega}_{\bs{c}}$  
is described by
\be
(x,y)\in \tilde{\Omega}_{\bs{c}} \quad \Longleftrightarrow \quad c_1x^{\alpha} < y < c_2x^{\alpha} \ \text{ and } \ x > 0.
\ee

We can assume that $\alpha > 1$. (If $\alpha=1$, then $\tilde\Omega_{\bs c}$ is in fact Lipschitz near $\bs c$, so there is nothing to prove; if $\alpha<1$ we exchange $x$ and $y$ to obtain a similar description of the boundary with $\alpha$ replaced by $1/\alpha$.) Now Lemma \ref{lem:MP} says that there exists a continuous extension operator $\mc{E} \colon W^{1,2}(\tilde{\Omega}_{\bs{c}})\rightarrow W^{1,2}(G)$ where $\tilde{\Omega}_{\bs{c}}\subset G$ and near $\bs{0}$ the domain ${G}$ is characterized by
\be
(x,y)\in G \quad \Longleftrightarrow \quad |y|<M x^{\alpha}  \ \text{ and } \ x > 0,
\ee 
with $M$ large enough. Since the boundary $\partial G$ is of class $C$, we can now infer by \cite[Theorem~4.17,~p.~267]{Evans:1987} and \cite[Theorem~1.4.2/1,~p.28]{MP01} that $W^{1,2}(G)$ satisfies statements (\ref{compprime}) and (\ref{denseprime}) of the proposition. In particular, $W^{1,2}(G)\rightarrow L^2(G)$ is compact and $C^1(\overline{G})$ is dense in $W^{1,2}(G)$. 

Using the fact that $\Phi$ is a $C^1$-diffeomorphism, it is easily shown that the pull-back map
\be
\label{map}
\Phi^* \colon W^{1,2}(\tilde{\Omega}_{\bs{c}})\rightarrow W^{1,2}({\Omega}_{\bs{c}}), \quad \phi\mapsto \phi\circ \Phi,
\ee
is well defined and bijective, with
\[
	\frac{1}{C'}\|\phi\|^2_{W^{1,2}(\tilde{\Omega}_{\bs{c}})}<\|\phi\circ \Phi\|^2_{W^{1,2}({\Omega}_{\bs{c}})}<C'\|\phi\|^2_{W^{1,2}(\tilde{\Omega}_{\bs{c}})}
\]
for some $C'>0$. Therefore, the inclusion $W^{1,2}(\Omega_{\bs c}) \to L^2(\Omega_{\bs c})$ can be written as the composition of a compact operator
\[
	W^{1,2}(\Omega_{\bs c}) \xrightarrow{(\Phi^{-1})^*} W^{1,2}(\tilde \Omega_{\bs c}) \xrightarrow{\ \mc{E}\ } W^{1,2}(G) \longrightarrow L^2(G)
\]
and a bounded operator
\[
	L^2(G) \longrightarrow L^2(\tilde \Omega_{\bs c}) \xrightarrow{\ \Phi^*\ } L^2(\Omega_{\bs c})
\]
(where the first map is restriction), and hence is compact. This completes the proof of \eqref{compprime} for regular Neumann domains.

To prove \eqref{denseprime}, let $\phi \in W^{1,2}(\Omega_{\bs c})$, so that $\phi \circ \Phi^{-1} \in W^{1,2}(\tilde \Omega_{\bs c})$ and $\mc{E}(\phi \circ \Phi^{-1}) \in W^{1,2}(G)$. For any $\delta > 0$, there exists $\tilde\phi \in C^1(\overline{G})$ with $\|\tilde\phi - \mc{E}(\phi \circ \Phi^{-1})\|_{W^{1,2}(G)} < \delta$, and hence
\begin{align*}
	\big\| \tilde\phi\big|_{\tilde\Omega_{\bs c}} \circ \Phi - \phi \big\|_{W^{1,2}(\Omega_{\bs c})} 
	&\leq C' \big\| \tilde\phi\big|_{\tilde\Omega_{\bs c}}  - \phi \circ \Phi^{-1} \big\|_{W^{1,2}(\tilde\Omega_{\bs c})} \\
	&\leq C' \big\| \tilde\phi  - \mc{E}(\phi \circ \Phi^{-1}) \big\|_{W^{1,2}(G)} \\
	& < C'\delta.
\end{align*}
Since $\tilde\phi\big|_{\tilde\Omega_{\bs c}} \circ \Phi \in C^1(\overline{\Omega_{\bs c}})$, this completes the proof of \eqref{denseprime}.

We next prove \eqref{compprime} for cracked Neumann domains, using the decomposition \eqref{isoemb}. More precisely, using Lemma \ref{c1neumannline} we may dissect $\Omega$ as in \eqref{dissect} and, without loss of generality, assume that the cusp is located on the boundary of $\Omega_{\mathrm{r}}$, as in Figure \ref{fig:Neumann-domains-schematic}. Note that
\[
	W^{1,2}(\Omega)\rightarrow W^{1,2}(\Omega_{\mathrm{l}})\oplus W^{1,2}(\Omega_{\mathrm{r}})
	\rightarrow L^2(\Omega_{\mathrm{l}})\oplus L^2(\Omega_{\mathrm{r}}) = L^2(\Omega)
\]
and so it is enough to prove compactness of the embedding $W^{1,2}(\Omega_\bullet)\rightarrow L^2(\Omega_\bullet)$ for $\bullet = \mathrm{l},\mathrm{r}$. For $\bullet = \mathrm{l}$ this follows from the Lipschitz property of $\partial\Omega_{\mathrm{l}}$. For $\bullet = \mathrm{r}$ we observe that $\partial\Omega_{\mathrm{r}}$ is Lipschitz except at the cusp, and so the proof given above for regular domains applies.

Finally, we prove \eqref{dense}. For $0< t <1$ sufficiently close to $1$ we have $\Omega_t^c\subset\Omega_\bullet$ for either $\bullet = \mathrm{l}$ or $\mathrm{r}$ (the case $\bullet = \mathrm{r}$ is shown in Figure \ref{fig:Neumann-domains-dissection}, so we choose $t$ sufficiently close to $1$ and $\epsilon>0$ small enough that $\Omega_{\bs c} = \Omega_{t}^c \subset \Omega_\bullet$. Now let $\phi\in W^{1,2}(\Omega)$. Given $\delta>0$, there exists by \eqref{denseprime} a function $\phi_\delta \in C^1(\overline{\Omega_{\bs c}})$ such that $\|\phi - \phi_\delta\|_{W^{1,2}(\Omega_{\bs c})}<\delta$. Choosing a smooth cutoff function $\chi$ that equals $1$ in $\Omega_{t+\epsilon}^c$ and vanishes in $\Omega_t$, we define $\tilde \phi_\delta = \chi \phi_\delta + (1-\chi) \phi \in W^{1,2}(\Omega)$ and compute
\be
\ba
\|\phi-\tilde{\phi}_\delta\|_{W^{1,2}(\Omega)}
&=\| \chi (\phi - \phi_\delta) \|_{W^{1,2}(\Omega_{\bs c})} \\
&\leq K \|\phi-\phi_\delta\|_{W^{1,2}({\Omega_{\bs c}})} \\ &< K\delta,
\ea
\ee  
where $K$ is a constant depending only on $\chi$.
%
Finally, since $\supp\chi \subset \overline{\Omega_t^c}$, we have
\[
	\tilde\phi_\delta \big|_{\overline{\Omega_t^c}} = \chi \phi_\delta\big|_{\overline{\Omega_t^c}} \in C^1(\overline{\Omega_t^c}).
\]
This completes the proof.
\end{proof}
\section{The Neumann Laplacian on a Neumann domain} 
\label{sec:NeumannLaplacian}

In this section we define the Neumann Laplacian on a Neumann domain $\Omega$, and establish some of its fundamental properties, in particular proving Theorems \ref{thm:discrete} and \ref{thm:restriction}. This relies on the technical results of the previous section, namely Proposition \ref{extens}.

\subsection{Definition and proof of Theorem \ref{thm:discrete}}
\label{sec:Neumanndef}
We define the Neumann Laplacian in the usual way, via a symmetric bilinear form.
\begin{defn}
\label{def:NeuLap}
The Neumann Laplacian on an open set $\Omega \subset M$, denoted $\DN$, is the unique self-adjoint operator corresponding to the bilinear form 
\begin{equation}
\label{form}
	a(\psi,\phi) := \int\limits_{\Omega}\langle\grad\psi, \grad\phi\rangle, \quad\mathcal{D}(a):=W^{1,2}(\Omega).
\end{equation}
\end{defn}
More precisely, $\DN$ is an unbounded operator on $L^2(\Omega)$, with domain
\begin{align}
\label{delta:domain}
	\mathcal{D}(\DN) = \big\{ \psi \in W^{1,2}(\Omega) : \exists f_\psi \in L^2(\Omega) \text{ with }   a(\psi,\phi) = \left<f_\psi, \phi\right>_{L^2(\Omega)} \forall \phi \in W^{1,2}(\Omega) \big\},
\end{align}
and for any $\psi \in \mathcal{D}(\DN)$ we have $\DN\psi = f_\psi$.
The existence and uniqueness of such an operator follows immediately from the completeness of the form domain $\mc{D}(a) = W^{1,2}(\Omega)$ and standard theory of self-adjoint operators, for instance \cite[Thm.~VIII.15]{ReeSim72}. If $\psi \in \mc{D}(\DN)$, then \eqref{delta:domain} implies
\[
	\int\limits_\Omega \left<\grad \psi, \grad \phi \right> = \int\limits_\Omega (\DN \psi) \phi
\]
for all $\phi \in W^{1,2}_0(\Omega)$, and hence $\Delta \psi = \DN \psi \in L^2(\Omega)$. That is, $\DN$ acts as the weak Laplace--Beltrami operator $\Delta$ defined in \eqref{weakBT}.

The next result is nontrivial, and relies on the special geometric structure of Neumann domains.

\begin{prop}
\label{prop:spectral-properties-of-domain}
If $\Omega \subset M$ is a Neumann domain for a Morse function, then $\DN$ has compact resolvent, and hence has purely discrete spectrum 
$\sigma(\DN)\subset [0,\infty)$. 
\end{prop}

\begin{proof}
Proposition~\ref{extens}\eqref{compprime} says that the form domain $W^{1,2}(\Omega)$ is compactly embedded in $L^2(\Omega)$, so the result follows from \cite[Thm.~XIII.64]{reed1978methods}.
\end{proof}

\subsection{Domain of the Neumann Laplacian}
\label{sec:domain}

We now describe the domain of the Neumann Laplacian, working towards the proof of Theorem \ref{thm:restriction}. Recalling the truncated domain $\Omega_t$ introducted in Section \ref{sec:truncate}, and the decomposition $\partial\Omega_t=\gamma_{\pm,t}\cup\gamma_{0,t}$ of its boundary in \eqref{gammapmdef}, we have the following.
\begin{prop}
\label{prop:domain_of_Laplacian} 
Let $\Omega$ be a Neumann domain of a Morse function $f$. The domain of the Neumann Laplacian is given by
\begin{equation}
\label{domain-1}
\ba
	\mathcal{D}(\DN) = \Big\{ \psi \in W^{1,2}(\Omega)  : \Delta& \psi \in L^2(\Omega),  \partial_{\nu}(\psi|_{\Omega_t})\big|_{{\gamma_{0,t}}^\mathrm{o}} = 0  \text{ for all } 0<t<1 \\
 & \text{and } \lim_{t \to 1} \int_{\gamma_{\pm,t}} \partial_{\nu}(\psi|_{\Omega_t}) \phi = 0 \text{ for all } \phi \in W^{1,2}(\Omega)  \Big\}
\ea
\end{equation}	
if $\Omega$ is regular, and
\be
\ba
\label{domain-1b}
\mathcal{D}(\DN) = \Big\{\psi &\in W^{1,2}(\Omega):  \Delta \psi \in L^2(\Omega), \ \partial_{\nu}(\psi|_{\Omega_t \cap \Omega_{\mathrm{l}}})\big|_{{\gamma_{0,t}}^\mathrm{o}} = \partial_{\nu}(\psi|_{\Omega_t \cap \Omega_{\mathrm{r}}})\big|_{{\gamma_{0,t}}^\mathrm{o}} = 0 \\
  & \text{ for all } 0<t<1, \text{ and }  \lim_{t \to 1} \int_{\gamma_{\pm,t}} \partial_{\nu}(\psi|_{\Omega_t}) \phi = 0 \text{ for all } \phi \in W^{1,2}(\Omega) \Big\}
\ea
\ee
if $\Omega$ is cracked.
\end{prop}
That is, to be in the domain of $\DN$, a function must satisfy Neumann boundary conditions on the Lipschitz part of the boundary, as well as a limiting boundary condition at each cusp. While this completely characterizes the domain of $\DN$, the limiting boundary conditions on $\gamma_{\pm,t}$ may be difficult to check in practice. Therefore, in the following section we will give simple criteria (Proposition \ref{prop:intlimit} and Corollary \ref{cor:intlimit}) which guarantee these limiting conditions are satisfied.

\begin{rem}
Our techniques actually give a more general result, not just valid for Neumann domains. The key points are that $\pO$ is Lipschitz except for a finite number of cusps and cracks, and the cracks admit a Lipschitz continuation. A stronger result will be given below (in Remark \ref{rem:approx}) that relies on the detailed structure of the cusps, which for Neumann domains is a consequence of Hartman's theorem.
\end{rem}

In proving the proposition, we must take into account the fact that $\Omega_\mathrm{l}$ and $\Omega_\mathrm{r}$ need not be Lipschitz; see 
Figure \ref{fig:Neumann-domains-dissection}, where $\Omega_\mathrm{r}$ has a cusp on its boundary. We therefore combine the dissection and truncation of Sections \ref{secdisec} and \ref{sec:truncate}, respectively. The resulting domains are shown in Figure \ref{fig:Neumann-domains-dissection-truncation}. Note that the boundaries of $\Omega_t \cap \Omega_\mathrm{l}$ and $\Omega_t \cap \Omega_\mathrm{r}$ can be partitioned into three parts: $\gamma_{\pm,t}$ coming from the truncation; $\tilde\eta$ coming from the dissection; and $\gamma_{0,t}$, coming from the original domain $\Omega$. We emphasize that the dissection \eqref{dissect} is an auxiliary construction, and our analysis does not depend on the specific choice of $\tilde{\eta}$.

Since $\eta\cup{\tilde{\eta}}$ has a Lipschitz neighbourhood in both $\Omega_t \cap \Omega_\mathrm{l}$ and $\Omega_t \cap \Omega_{\mathrm{r}}$, cf. Figure \ref{fig:Neumann-domains-dissection-truncation}, we have
\be
\label{eqint2}
	(\phi|_{\Omega_{\mathrm{l}}})\big|_{\tilde{\eta}^{\mathrm o}} = (\phi|_{\Omega_{\mathrm{r}}})\big|_{\tilde{\eta}^{\mathrm o}} \in H^{\frac12}(\tilde\eta^{\mathrm o}) \text{ for } \phi\in W^{1,2}(\Omega),
\ee
with ${\cdot}^{\mathrm{o}}$ denoting the interior in $\gamma_{0,t}\cup\tilde{\eta}$. Therefore, the map
\be
\label{modnormal}
\ba
W^{1,2}(\Omega) &\rightarrow H^{\frac12}({\eta}^{\mathrm{o}})\oplus H^{\frac12}({\eta}^{\mathrm{o}})\oplus H^{\frac12}(\tilde{\eta}^{\mathrm o}) \\
\phi &\mapsto \big( (\phi|_{\Omega_{\mathrm{l}}})|_{{\eta}^{\mathrm{o}}}, (\phi|_{\Omega_{\mathrm{r}}})|_{{\eta}^{\mathrm{o}}}, \phi|_{\tilde{\eta}^{\mathrm o}} \big)
\ea
\ee
is well defined, where $\phi|_{\tilde{\eta}^{\mathrm o}}$ denotes the common value in \eqref{eqint2}. We first analyze the normal derivatives on $\tilde{\eta}$.

\begin{lem}
\label{auxlem1}
Let $\Omega$ be a cracked Neumann domain. If $\psi\in W^{1,2}(\Omega)$ and $\Delta \psi \in L^2(\Omega)$, then
\be
\label{eqint3}
	\partial_{\nu}(\psi|_{\Omega_{\mathrm{l}}})\big|_{\tilde{\eta}^{\rm o}} + \partial_{\nu}(\psi|_{\Omega_{\mathrm{r}}})\big|_{\tilde{\eta}^{\rm o}} = 0 \in H^{-\frac12}(\tilde{\eta}^{\rm o}).
\ee

\end{lem}

\begin{proof}
The hypothesis $\Delta \psi \in L^2(\Omega)$ means
\be
	\int\limits_{\Omega}\langle\grad \psi,\grad\phi\rangle = \langle\Delta \psi,\phi\rangle_{L^2(\Omega)}
\ee
for all $\phi \in C^\infty_0(\Omega)$. Together with Green's formula \eqref{ggsmooth}, this implies
\begin{equation}
\label{oneta}
	\int_{\tilde{\eta}}\big(\partial_{\nu}(\psi|_{\Omega_{\mathrm{l}}}) + \partial_{\nu}(\psi|_{\Omega_{\mathrm{r}}}) \big) \phi = 0
\end{equation}
for all $\phi \in C^\infty_0(\Omega)$. Since any function in $C^\infty_0(\tilde\eta)$ can be realized as $\phi|_{\tilde\eta}$ for some $\phi \in C^\infty_0(\Omega)$, the result follows from \eqref{equality}.
\end{proof}

\begin{figure}[ht]
\centering{}\includegraphics[width=0.6\textwidth]{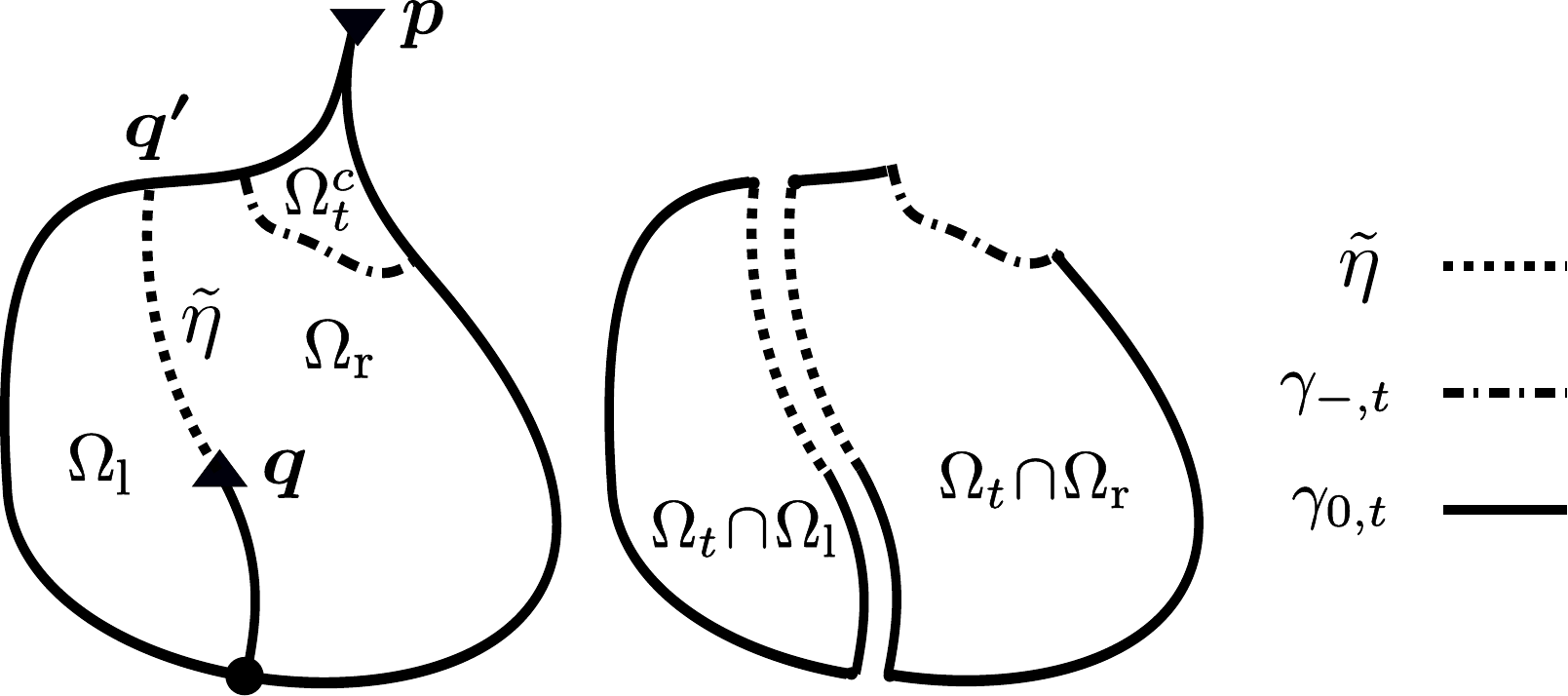}
\caption{The dissected and truncated domains appearing in the proof of Proposition \ref{prop:domain_of_Laplacian}. Here $\gamma_{\pm,t}$ is a result of the truncation, $\tilde\eta$ is from the dissection, and $\gamma_{0,t}$ is the part of the original boundary, $\partial\Omega$, that remains after the truncation. }
\label{fig:Neumann-domains-dissection-truncation}
\end{figure}

We next analyze the normal derivative on the Lipschitz part of the boundary, $\gamma_{0,t}$.

\begin{lem}
\label{auxlem2}
If $\psi \in \mathcal{D}(\Delta^N_{\Omega})$, then
\be
\label{neubound}
	\partial_{\nu}(\psi|_{\Omega_t \cap \Omega_\bullet})\big|_{\gamma_{0,t}^\mathrm{o}} = 0 \in H^{-\frac12}\big((\partial(\Omega_t\cap\Omega_\bullet) \cap \gamma_{0,t})^\mathrm{o}\big)
\ee
for any $0 < t < 1$, where $\bullet = \mathrm{l,r}$.
\end{lem}
\begin{proof} 
We prove the result for $\Omega_{\mathrm l}$, the argument for $\Omega_{\mathrm r}$ is identical. For any test function $\phi\in W^{1,2}(\Omega_t \cap \Omega_{\mathrm l})$ with $\phi|_{\tilde\eta} = 0$ and $\phi|_{\gamma_{+,t}} = 0$, we get from Green's formula \eqref{ggsmooth} that
\[
	\int\limits_{\partial(\Omega_t \cap \Omega_{\mathrm l}) \cap \gamma_{0,t}} \partial_{\nu}(\psi|_{\Omega_t \cap \Omega_{\mathrm{l}}})\phi = 0.
\]
The image of the trace map restricted to $\{\phi\in W^{1,2}(\Omega_t \cap \Omega_{\mathrm l}) : \phi|_{\tilde\eta} = \phi|_{\gamma_{+,t}} = 0 \}$ is precisely $\widetilde{H}^{\frac12}\big((\partial(\Omega_t\cap\Omega_{\mathrm l}) \cap \gamma_{0,t})^\mathrm{o}\big)$, by \eqref{sobolev2} and \cite[Thm.~3.37]{McL}, so the result follows.
\end{proof}
Now, equipped with our preliminary analysis of normal derivatives, we prove Proposition~\ref{prop:domain_of_Laplacian}.
\begin{proof}[Proof of Proposition \ref{prop:domain_of_Laplacian}]
We only prove \eqref{domain-1b}; the proof of \eqref{domain-1} for regular domains is similar but less involved, so we omit it. Let $\psi \in W^{1,2}(\Omega)$. From \eqref{delta:domain} we have that $\psi \in \mc{D}(\DN)$ if and only if $\Delta \psi \in L^2(\Omega)$ and
\be
\label{Ndomain1}
	\int\limits_\Omega (\Delta \psi)\phi = \int\limits_{\Omega}\langle\grad \psi,\grad\phi\rangle
\ee
for all $\phi \in W^{1,2}(\Omega)$. Thus, we fix $\psi,\phi \in W^{1,2}(\Omega)$ with $\Delta \psi \in L^2(\Omega)$. Since $\langle\grad \psi,\grad\phi\rangle$ and $(\Delta\psi)\phi$ are in $L^1(\Omega)$, their integrals over $\Omega_t$ converge to their integrals over $\Omega$ as $t \to 1$ by the dominated convergence theorem, hence \eqref{Ndomain1} is equivalent to
\be
\label{Ndomain2}
	\lim_{t \to 1} \int\limits_{\Omega_t} (\Delta \psi) \phi = \lim_{t \to 1} \int\limits_{\Omega_t}\langle\grad \psi,\grad\phi\rangle.
\ee
We now use the dissection \eqref{dissect}, applying Green's formula on the truncated and dissected domains to obtain
\begin{align*}
	\int\limits_{\Omega_t} \langle\grad \psi,\grad\phi\rangle
	&= \int\limits_{\Omega_t} (\Delta \psi) \phi + 
	\int\limits_{\partial(\Omega_t\cap\Omega_{\mathrm l})} \partial_{\nu}(\psi|_{\Omega_t \cap \Omega_{\mathrm l}}) \phi +  
	\int\limits_{\partial(\Omega_t\cap\Omega_{\mathrm r})} \partial_{\nu}(\psi|_{\Omega_t \cap \Omega_{\mathrm r}}) \phi.
\end{align*}
Comparing with \eqref{Ndomain2}, we see that $\psi \in \mc{D}(\DN)$ if and only if
\be
\label{Ndomain3}
	\lim_{t \to 1} \left\{  \int\limits_{\partial(\Omega_t\cap\Omega_{\mathrm l})} \partial_{\nu}(\psi|_{\Omega_t \cap \Omega_{\mathrm l}}) \phi +  
	\int\limits_{\partial(\Omega_t\cap\Omega_{\mathrm r})} \partial_{\nu}(\psi|_{\Omega_t \cap \Omega_{\mathrm r}}) \phi \right\} = 0
\ee
for each $\phi \in W^{1,2}(\Omega)$. Therefore, it suffices to show that \eqref{Ndomain3} is equivalent to the conditions in \eqref{domain-1b}.

We claim that if $\psi\in W^{1,2}(\Omega)$ satisfies $\Delta \psi \in L^2(\Omega)$ and $\partial_{\nu}(\psi|_{\Omega_t \cap \Omega_{\mathrm{l}}})\big|_{{\gamma_{0,t}}^\mathrm{o}} = \partial_{\nu}(\psi|_{\Omega_t \cap \Omega_{\mathrm{r}}})\big|_{{\gamma_{0,t}}^\mathrm{o}} = 0$, then
\be
\label{Ndomain4}
\int\limits_{\partial(\Omega_t\cap\Omega_{\mathrm l})} \partial_{\nu}(\psi|_{\Omega_t \cap \Omega_{\mathrm l}}) \phi +  
	\int\limits_{\partial(\Omega_t\cap\Omega_{\mathrm r})} \partial_{\nu}(\psi|_{\Omega_t \cap \Omega_{\mathrm r}}) \phi
	= \int_{\gamma_{\pm,t}} \partial_{\nu}(\psi|_{\Omega_t \cap \Omega_{\rm r}}) \phi
\ee
for any $0 < t < 1$ and $\phi \in W^{1,2}(\Omega)$. To prove this, we decompose the integrals over 
$\partial(\Omega_t\cap\Omega_{\mathrm l})$ and $\partial(\Omega_t\cap\Omega_{\mathrm r})$ into a sum of integrals over the different parts of the boundary. This is nontrivial, since this integral notation actually represents the action of the normal derivative distribution on a test function in $H^{\frac12}$, and a distribution in $H^{-\frac12}$ does not necessarily split into the sum of its restriction to different parts of the boundary, as discussed in Section \ref{sec:Sobolevdef}.

Here we make use of Lemma \ref{lem:split}, as well as the assumption that $\partial_{\nu}(\psi|_{\Omega_t \cap \Omega_{\mathrm{l}}})\big|_{{\gamma_{0,t}}^\mathrm{o}} = \partial_{\nu}(\psi|_{\Omega_t \cap \Omega_{\mathrm{r}}})\big|_{{\gamma_{0,t}}^\mathrm{o}} = 0$. Applying the lemma to $N = \Omega_t \cap \Omega_{\rm l}$, with the boundary decomposed into $\Gamma_1 = \tilde\eta$, $\Gamma_2 = \gamma_{+,t}$ and $\Gamma_0 = \partial(\Omega_t \cap \Omega_{\rm l}) \cap \gamma_{0,t}$ we obtain
\[
	\int\limits_{\partial(\Omega_t\cap\Omega_{\mathrm l})} \partial_{\nu}(\psi|_{\Omega_t \cap \Omega_{\mathrm l}}) \phi = \int\limits_{\tilde\eta} \partial_{\nu}(\psi|_{\Omega_t \cap \Omega_{\mathrm l}}) \phi + \int\limits_{\gamma_{+,t}} \partial_{\nu}(\psi|_{\Omega_t \cap \Omega_{\mathrm l}}) \phi.
\]
Similarly, for $\Omega_t \cap \Omega_{\rm r}$ we get
\[
	\int\limits_{\partial(\Omega_t\cap\Omega_{\mathrm r})} \partial_{\nu}(\psi|_{\Omega_t \cap \Omega_{\mathrm r}}) \phi = \int\limits_{\tilde\eta} \partial_{\nu}(\psi|_{\Omega_t \cap \Omega_{\mathrm r}}) \phi + \int\limits_{\gamma_{-,t}} \partial_{\nu}(\psi|_{\Omega_t \cap \Omega_{\mathrm r}}) \phi.
\]

Adding these together and using Lemma \ref{auxlem1} to cancel the $\tilde\eta$ terms completes the proof of \eqref{Ndomain4}.

To finish the proof of the proposition, suppose that $\psi \in \mc{D}(\DN)$, so \eqref{Ndomain3} holds. Lemma~\ref{auxlem2} implies that $\partial_{\nu}(\psi|_{\Omega_t \cap \Omega_{\mathrm{l}}})\big|_{{\gamma_{0,t}}^\mathrm{o}} = \partial_{\nu}(\psi|_{\Omega_t \cap \Omega_{\mathrm{r}}})\big|_{{\gamma_{0,t}}^\mathrm{o}} = 0$, so we can use \eqref{Ndomain4} to conclude that
\[
	\lim_{t \to 1} \int_{\gamma_{-,t}} \partial_{\nu}(\psi|_{\Omega_t \cap \Omega_{\rm r}}) \phi = 0
\]
for any $\phi\in W^{1,2}(\Omega)$. Therefore, the boundary conditions given in \eqref{domain-1b} are satisfied. Conversely, if $\psi$ satisfies the boundary conditions in \eqref{domain-1b}, we take the limit of \eqref{Ndomain4} to find that \eqref{Ndomain3} holds and so $\psi \in \mc{D}(\DN)$.
\end{proof}

\begin{rem}
\label{rem:approx}
For the $\gamma_{\pm,t}$ boundary condition in \eqref{domain-1} or \eqref{domain-1b}, it is enough to check that
\begin{equation}
\label{intlimit}
	\lim_{t \to 1} \int_{\gamma_{\pm,t}} \partial_{\nu}(\psi|_{\Omega_t}) \phi = 0 
\end{equation}
for test functions $\phi \in W^{1,2}(\Omega)$ that are $C^1$ in a neighbourhood of the cusp. If $\psi$ satisfies this, and the other conditions in \eqref{domain-1} or \eqref{domain-1b}, the proof of Proposition \ref{prop:domain_of_Laplacian} shows that \eqref{Ndomain1} holds for all such $\phi$. It then follows from Proposition \ref{extens}(\ref{dense}) that \eqref{Ndomain1}  in fact holds for all $\phi \in W^{1,2}(\Omega)$, and so $\psi \in \mc D(\DN)$.
\end{rem}

\subsection{Proof of Theorem \ref{thm:restriction}}
If $\Omega$ has no cusps or cracks, then Proposition \ref{prop:domain_of_Laplacian} says that $\mc{D}(\DN)$ simply consists of functions that are sufficiently regular and satisfy Neumann boundary conditions everywhere on $\partial\Omega$. On the other hand, when a cusp is present we must also impose the condition \eqref{intlimit}, which says the normal derivative of $\psi$ on $\gamma_{\pm,t}$ does not blow up as the cusp is approached. We now give a simple condition that guarantees this is the case. 

For simplicity we only state the result for a cusp at $\bs q$; the corresponding statement for a cusp at $\bs p$ is analogous. We define the ``doubly-truncated domain"
\be
	\Omega'_t  =  \{x \in \Omega : t_0 f(\bs q) < f(x) < t f(\bs q)  \}
\ee
for a fixed $0 < t_0 < 1$.

\begin{prop}
\label{prop:intlimit}
If $\psi \in W^{1,2}(\Omega)$, and there exists $t_0$ such that $\psi \in W^{2,2}(\Omega_t')$ for all $t_0 < t < 1$ and $(1-t)^{1/2} \| \psi \|^2_{W^{2,2}(\Omega_t')}$ is bounded near $t=1$, then \eqref{intlimit} holds.
\end{prop}

The proposition does not assume $\psi$ is in $W^{2,2}(\Omega)$, but only that its $W^{2,2}(\Omega_t')$ norm does not blow up too quickly near the cusp. Of course this condition is automatically satisfied if $\psi \in W^{2,2}(\Omega)$.

\begin{cor}
\label{cor:intlimit}
If $\psi \in W^{2,2}(\Omega)$, then \eqref{intlimit} holds.
\end{cor}

Since the Morse function $f$ that generated the Neumann domain $\Omega$ was assumed to be smooth, $f\big|_\Omega$ satisfies the hypotheses of Corollary \ref{cor:intlimit}, and Theorem \ref{thm:restriction} follows immediately.


The main ingredient in the proof is a trace estimate for the doubly-truncated domain $\Omega_t'$, with controlled dependence on $t$.

\begin{lem}
\label{lem:tracet}
There exist constants $A,B>0$ such that
\begin{equation}
\label{quant:trace}
	\int_{\gamma_{\pm,t}} u^2 \leq \frac{A}{\sqrt{1-t}} \int_{\Omega_t'} u^2 + B \int_{\Omega_t'} |\nabla u|^2 
\end{equation}
for all $u \in W^{1,2}(\Omega_t')$ and $t$ sufficiently close to $1$.
\end{lem}
\begin{proof}
Increasing $t_0$ if necessary, we can assume that $\bs q$ is the only critical point of $f$ in $\overline{\Omega_{t_0}^c}$.
Consider the vector field
\[
	X := \chi u^2 \frac{\nabla f}{|\nabla f|},
\]
where $\chi$ is a smooth cutoff function that vanishes on $\gamma_{+,t_0}$ and equals $1$ in a neighbourhood of $\bs q$. Since $f$ is smooth and has no critical points in $\overline{\Omega'_t}$, we have $X \in W^{1,1}(\Omega'_t)$. Observe that $\nabla f / |\nabla f|$ is tangent to $\gamma_{0,t}$, whereas on $\gamma_{+,t}$ it coincides with the outward unit normal. This implies
\[
	\int_{\partial \Omega'_t} X \cdot \nu = \int_{\gamma_{+,t}} u^2
\]
for any $t$ large enough that $\chi|_{\gamma_{+,t}} \equiv 1$. On the other hand, the divergence theorem implies
\[
	\int_{\partial \Omega'_t} X \cdot \nu = \int_{\Omega'_t} \dv X = \int_{\Omega'_t} \left( \nabla(\chi  u^2) \cdot \frac{\nabla f}{|\nabla f|} + \chi u^2 \dv \frac{\nabla f}{|\nabla f|} \right),
\]
so we obtain
\begin{equation}
\label{b:est2}
	\int_{\gamma_{+,t}} u^2 \leq  B \| u\|^2_{H^1(\Omega'_t)} + \int_{\Omega'_t} u^2 \left|\dv \frac{\nabla f}{|\nabla f|} \right|
\end{equation}
for some constant $B$ depending only on $\chi$.

To estimate the integral on the right-hand side, we observe that the level sets of $f$ have mean curvature $\dv (\nabla f / |\nabla f|)$. Using the Morse lemma, we can find coordinates $(x,y)$ in a neighbourhood of $\bs q$ such that $f(x,y) = f(\bs q) - x^2 - y^2$. A straightforward computation (c.f. \cite[Lem. 4.7]{BCJLS21}) gives
\[
	\left|\dv \frac{\nabla f}{|\nabla f|} (x,y) \right| \leq \frac{C}{\sqrt{x^2 + y^2}} = \frac{C}{\sqrt{f(\bs q) - f(x,y)}}
\]
and so we have the uniform estimate
\[
	\left|\dv \frac{\nabla f}{|\nabla f|} \right| \leq \frac{C}{\sqrt{f(\bs q)(t-1)}}
\]
on $\overline{\Omega_t'}$. Substituting this into \eqref{b:est2} completes the proof.
\end{proof}
The other ingredient in the proof of Proposition \ref{prop:intlimit} is the following geometric estimate.
\begin{lem}
\label{lem:length}
The length of $\gamma_{\pm,t}$ is $o\big((1-t)^{1/2}\big)$ near $t=1$.
\end{lem}

\begin{proof}
We prove the result for $\gamma_{+,t}$, assuming there is a cusp at the maximum $\bs{q}$; the proof for $\gamma_{-,t}$ is identical. Using the Morse lemma, we can find coordinates $(x,y)$ near $\bs q$ such that $f(x,y) = f(\bs q) - x^2 - y^2$, and so $\gamma_{+,t}$ is contained in the circle of radius $\rho = \sqrt{(1-t)f(\bs q)}$. More precisely, it is the arc bounded by the angles $\theta_1(t)$ and $\theta_2(t)$. Parameterizing this as $\gamma(\theta) = (\rho \cos\theta, \rho \sin\theta)$, we have $|\gamma'(\theta)|_g \leq C \sqrt{1-t}$, where $|\cdot|_g$ denotes the length computed using the metric $g$, and $C$ is some constant depending on $f(\bs q)$ and the components of $g$ in this coordinate chart. This implies
\[
	L(\gamma_{+,t}) = \int_{\theta_1(t)}^{\theta_2(t)} |\gamma'(\theta)|_g \,d\theta \leq C \sqrt{1-t} \big|\theta_2(t) - \theta_1(t)\big|.
\]
Near $\bs q$, the boundary $\pO$ consists of two Neumann lines meeting tangentially at $\bs q$ (since there is a cusp). This implies $|\theta_2(t) - \theta_1(t)| \to 0$ as $t \to 1$ and completes the proof.
\end{proof}

We are now ready to prove Proposition \ref{prop:intlimit}.

\begin{proof}
Since $|\partial_\nu \psi| \leq |\nabla \psi|$, it is enough to show that
\begin{equation}
\label{intlimit2}
	\lim_{t \to 1} \int_{\gamma_{\pm,t}} |\nabla \psi| \phi = 0
\end{equation}
for all $\phi \in W^{1,2}(\Omega)$ that are $C^1$ in a neighbourhood of $\bs q$; see Remark \ref{rem:approx}. Fix such a $\phi$ and define $u = \chi |\nabla \psi| \phi$, where $\chi$ is a smooth cutoff function that equals $1$ near $\bs q$ and is supported in the region where $\phi$ is smooth. The hypotheses on $\psi$ imply $u \in L^2(\Omega)$ and $u \in W^{1,2}(\Omega_t')$ for all $t_0 < t<1$, with
\[
	\|u\|_{L^2(\Omega)} \leq C \| \psi \|_{W^{1,2}(\Omega)}, \qquad \|u\|_{W^{1,2}(\Omega_t')} \leq C \| \psi \|_{W^{2,2}(\Omega_t')},
\]
for some constant $C$ depending only on $\phi$ and $\chi$.

Using H\"older's inequality and Lemma~\ref{lem:tracet}, we obtain
\begin{align*}
	\left( \int_{\gamma_{\pm,t}} u \right)^2 &\leq \left(\int_{\gamma_{\pm,t}} u^2 \right) L(\gamma_{\pm,t}) \\
	& \leq \left(\frac{A}{\sqrt{1-t}} \|u\|_{L^2(\Omega_t)}^2 + B \|u\|_{W^{1,2}(\Omega_t')}^2 \right) L(\gamma_{+,t}) \\
	& \leq \left(A \|u\|_{L^2(\Omega)}^2 + B \sqrt{1-t} \|u\|_{W^{1,2}(\Omega_t')}^2 \right) \frac{L(\gamma_{+,t})}{\sqrt{1-t}}.
\end{align*}
By Lemma \ref{lem:length} this tends to zero as $t \to 1$.
\end{proof}

\subsection*{Acknowledgments}{}
G.C. acknowledges the support of NSERC grant RGPIN-2017-04259. R.B. and S.K.E. were supported by ISF (grant No. 844/19).

\appendix

\section{Morse--Smale functions with cracked Neumann domains}
\label{app:typeii}

In this appendix we construct Morse--Smale functions having cracked Neumann domains.
As in the rest of the paper, we assume $M$ is a smooth, closed, connected orientable surface.

\begin{thm}
\label{thm:typeii}
Let $f$ be a Morse--Smale function on $M$ and $\Omega$ a Neumann domain of $f$. Then there exists a Morse--Smale function $\tilde f$ that has a cracked Neumann domain $\widetilde\Omega \subset \Omega$.
\end{thm}

We will see in the proof that $\tilde f$ can be chosen to agree with $f$ outside an arbitrary open set $U \subset \Omega$. However, the difference $\tilde f - f$ may be large inside $U$. The existence of $\tilde f$ is given by the following general lemma.

\begin{lem}
\label{lem:fperturb}
Let $U \subset M$ be an open subset, and $f : U \to \R$ a smooth function having no critical points. There exists a smooth function $\tilde f : U \to \R$, with $\supp\,(\tilde f - f) \subset U$, whose only critical points are a non-degenerate maximum and a non-degenerate saddle.
\end{lem}

\begin{proof}
Since $f$ has no critical points in $U$, we can invoke the canonical form theorem for smooth vector fields and find local coordinates $(x,y)$ with respect to which $f(x,y) = Ax+B$, for $(x,y) \in (-1,1) \times (-1,1)$.  Now choose a smooth function $\alpha(x)$ with $\supp\, \alpha \subset (-1,1)$ and
\begin{align}
	\int_{-1}^1 \alpha(x)\,dx &= 0\,,
\end{align}
so that there exist points $-1 < x_1 < x_2 < 1$ with
\begin{align}
	\alpha(x) \quad \begin{cases} > -A, & -1 < x < x_1\,, \\
	= -A, & x = x_1\,, \\
	< -A, & x_1 < x < x_2\,, \\
	= -A, & x = x_2\,, \\
	> -A, & x_2 < x < 1\,,
	\end{cases}
\end{align}
as shown in Figure \ref{fig:alpha}.

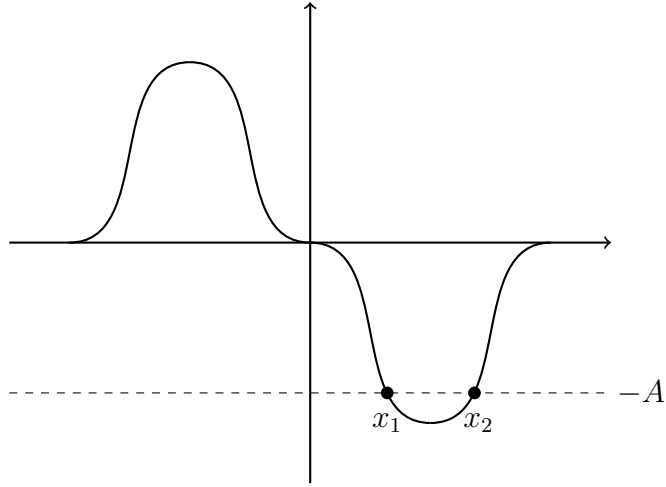
\begin{figure}
\begin{tikzpicture}[scale=0.8]
	\draw[thick,->] (-5,0) to (5,0); 
	\draw[thick,->] (0,-4) to (0,4); 
	\draw[thick] (-4,0) to [out=0, in=180] (-2,3); 
	\draw[thick] (-2,3) to [out=0, in=180] (0,0); 
	\draw[thick] (0,0) to [out=0, in=180] (2,-3); 
	\draw[thick] (2,-3) to [out=0, in=180] (4,0); 
	\draw[dashed] (-5,-2.5) to (5,-2.5); 
	\node at (5.5,-2.5) {$-A$}; 
	\fill (1.28,-2.5) circle[radius=3pt];
	\fill (2.73,-2.5) circle[radius=3pt];
	\node at (1.28,-3) {$x_1$}; 
	\node at (2.8,-3) {$x_2$}; 
\end{tikzpicture}
\caption{The function $\alpha(x)$ used in the proof of Lemma \ref{lem:fperturb}.}
\label{fig:alpha}
\end{figure}

We define
\begin{equation}
	\tilde f(x,y) = f(x,y) + \beta(x) \gamma(y)\,,
\end{equation}
where $\beta(x) = \int_{-1}^x \alpha(t)\,dt$ and $\gamma(y)=\exp\{-1/(1-y^2)\}$. Note that $\gamma$ is a non-negative bump function supported in $(-1,1)$ with $\gamma'(0) = 0$ and $\gamma''(0) < 0$. It follows that
\[
	\frac{\partial \tilde f}{\partial x} = A + \alpha(x)\gamma(y) \quad \text{ and } \quad
	\frac{\partial \tilde f}{\partial y} = \beta(x) \gamma'(y)\,,
\]
and so the only critical points of $\tilde f$ in $U$ are $(x_1,0)$ and $(x_2,0)$. We compute
\begin{align*}
	\frac{\partial^2 \tilde f}{\partial x^2}(x_1,0) &= \alpha'(x_1)\gamma(0) < 0\,, \\
	\frac{\partial^2 \tilde f}{\partial x^2}(x_2,0) &= \alpha'(x_2)\gamma(0) > 0\,, \\
	\frac{\partial^2 \tilde f}{\partial y^2}(x_i,0) &= \beta(x_i) \gamma''(0) < 0\,,
\end{align*}
and conclude that $(x_1,0)$ and $(x_2,0)$ are a non-degenerate maximum and a non-degenerate saddle, respectively.
\end{proof}
\begin{proof}[Proof of Theorem \ref{thm:typeii}]
If $\Omega$ is cracked we simply choose $\tilde f = f$ and there is nothing to prove.
Therefore we assume that $\Omega$ is regular. Since $f$ is Morse--Smale, Theorem \ref{thm:topological-properties-manifolds} says the closure of $\Omega$ contains exactly four critical points, all of which are on the boundary: a maximum $\bs{q}$, a minimum $\bs{p}$, and saddle points $\bs{r}_1$ and $\bs{r}_2$; see Figure \ref{fig:Neumann-domains-schematic}.

Now choose $\tilde f$ according to Lemma \ref{lem:fperturb}, for some open set $U \Subset \Omega$. By construction, $\tilde f$ has two critical points in $\Omega$: a maximum $\bs{q_*}$ and a saddle point $\bs{r_*}$. Since $\tilde f$ is a Morse function, $\bs{r_*}$ has degree four, i.e. there are four Neumann lines connected to $\bs r_*$. We obtain the result by studying the endpoints of these lines, as depicted in Figure \ref{fig:Neumann-domains-contradiction}.
\begin{figure}[ht]
\centering{}\includegraphics[width=0.6\textwidth]{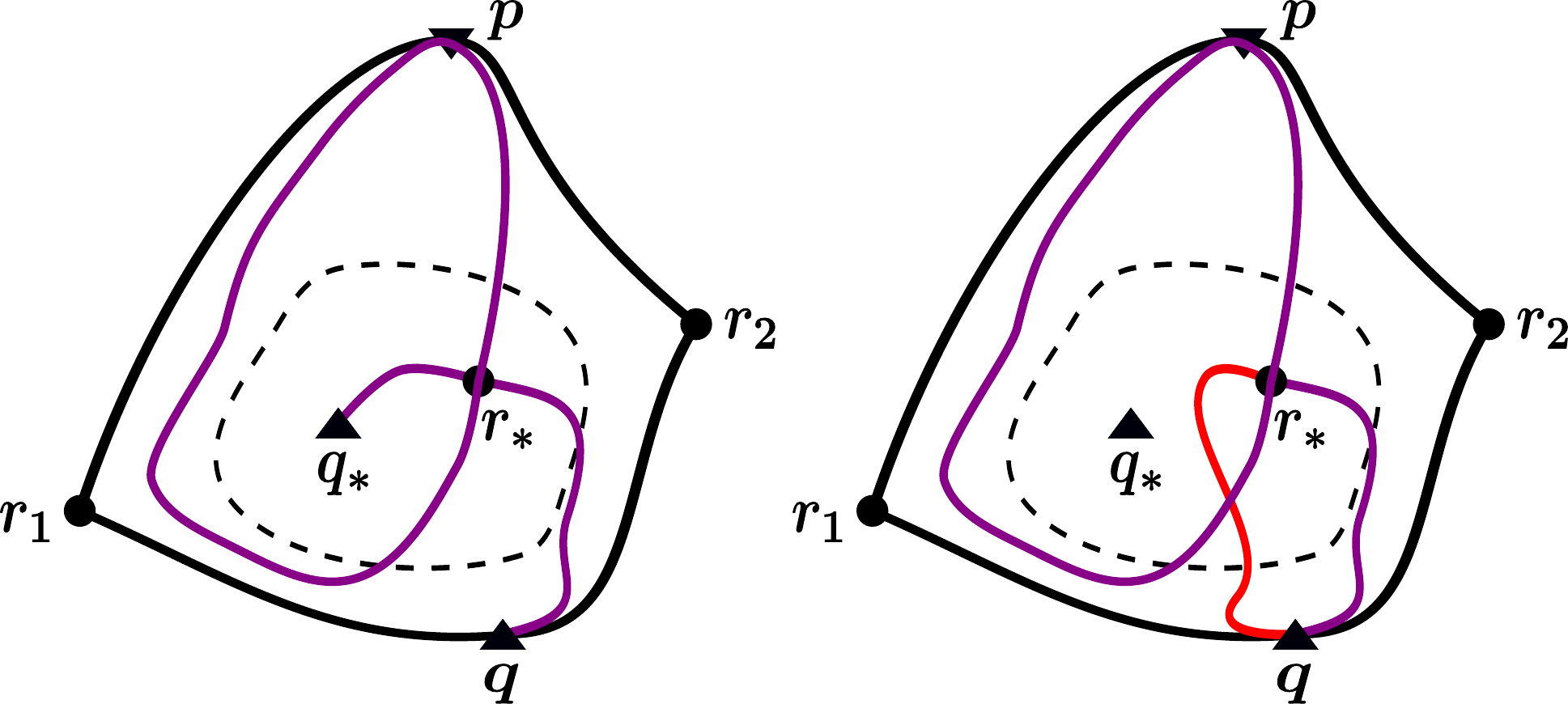} \caption{Left: the cracked Neumann domain constructed in Theorem \ref{thm:typeii}, with Neumann lines shown in purple. Right:
if two Neumann lines connected $\bs r_*$ to $\bs q$, one of them (shown in red) would have to intersect another Neumann line, which is impossible.
%
The dashed line represents the boundary of the set $U$ containing $\supp\,(\tilde f - f)$.
}
\label{fig:Neumann-domains-contradiction}
\end{figure}
Since $\tilde f$ agrees with $f$ in a neighbourhood of $\partial\Omega$, the invariant manifolds $W^s(\bs r_i)$ and $W^u(\bs r_i)$ are unchanged by the perturbation. As a result, it is not possible for any of the Neumann lines coming from $\bs r_*$ to end at $\bs r_1$ or $\bs r_2$. Therefore, the four Neumann lines from $\bs r_*$ can only end at $\bs q$, $\bs p$ or $\bs q_*$, so it follows from Lemma \ref{prop:Morse-Smale-on-2d} that $\tilde f$ is Morse--Smale. The two lines along which $\tilde f$ is decreasing must end at $\bs p$, since it is the only minimum in $\overline\Omega$. This means the two lines along which $f$ is increasing are connected to either $\bs q$ or $\bs q_*$. We claim that there is one Neumann line connected to each maximum.

Suppose instead that both ended at $\bs q$. Then the union of these Neumann lines forms a closed loop. Similarly, the union of the two lines ending at $\bs p$ is a closed loop. Both loops intersect at $\bs r_*$, where they are orthogonal by Proposition \ref{prop:angles-at-critical-pts}\eqref{enu:prop-angles-at-critical-pts-1}. Since $\Omega$ is simply connected, this can only happen if the loops also intersect at a point other than $\bs r_*$, but this is impossible since gradient flow lines cannot cross. The same argument shows that these lines cannot both be connected to $\bs q_*$, hence one must end at each maximum.

Since all of the Neumann lines in $\overline\Omega$ have been accounted for, this means $\bs q_*$ has degree one, hence the Neumann domain with $\bs q_*$ on its boundary is cracked.
\end{proof}

{\small{}\bibliographystyle{amsalpha}
\bibliography{Literature}
}{\small\par}
\end{document}